\documentclass[11pt]{article}

\usepackage{amsfonts}
\usepackage{amscd}
\usepackage{amssymb}
\usepackage{amsthm}
\usepackage{amsmath}
\usepackage{stmaryrd}
\usepackage{graphicx}
\usepackage{color}
\usepackage{verbatim}
\usepackage{mathrsfs}
\usepackage{dsfont}
\usepackage{tikz}
\usepackage{subfigure}

 \theoremstyle{plain}
\newtheorem{thm}{Theorem}[section]
\newtheorem{lemma}[thm]{Lemma}
\newtheorem{prop}[thm]{Proposition}

\theoremstyle{definition}
\newtheorem{Example}{Example}

\newtheorem{remark}[thm]{Remark}
\theoremstyle{remark}

\numberwithin{equation}{section}

\def\CC{\mathbb{C}}

\def\FF{\mathbb{F}}

\def\QQ{\mathbb{Q}}

\def\TT{\mathbb{T}}

\def\ZZ{\mathbb{Z}}

\def\scH{\mathscr{H}}

\def\dim{\mathrm{dim}}

\def\Hom{\mathrm{Hom}}

\def\Tr{\mathrm{Tr}}

\definecolor{Gray}{gray}{0.5}

\DeclareMathOperator*{\Res}{Res}

\makeatletter
\renewcommand{\@makefnmark}{\mbox{\textsuperscript{}}}
\makeatother

\title{On calibrated representations and the Plancherel Theorem for affine Hecke algebras}
\author{James Parkinson\footnote{Partially supported under the Australian Research Council (ARC) discovery grant DP110103205.}}
\date{November 28, 2013}

\begin{document}

\maketitle

\begin{abstract}
This paper has two main purposes. Firstly we generalise Ram's combinatorial construction \cite{ramcalib} of calibrated representations of the affine Hecke algebra to the multi-parameter case (including the non-reduced $BC_n$ case). We then derive the Plancherel formulae for all rank~$1$ and rank~$2$ affine Hecke algebras, using our calibrated representations to construct all representations involved. 
\end{abstract}



\section*{Introduction}

In this paper we extend Ram's combinatorial construction of calibrated representations of affine Hecke algebras to the multi-parameter case (including the non-reduced case), and we use these representations to derive explicit Plancherel formulae for all rank~$1$ and rank~$2$ affine Hecke algebras, following the work of Opdam (\cite{opdamtrace}, \cite{opdamhanalysis}).

Let us discuss the relevance and significance of each of these objectives. Affine Hecke algebras arise in the study of representation theory of groups~$G$ of Lie type defined over local fields such as $\FF_q((t))$ or $\QQ_p$. If $I$ is an \textit{Iwahori subgroup} of $G$ then complex representations of~$G$ with vectors fixed by $I$ can be studied via corresponding representations of the associated affine Hecke algebra $\scH=\mathcal{C}_c(I\backslash G/I)$ of continuous compactly supported $I$ bi-invariant complex valued functions on~$G$ (see \cite{borel}, \cite{matsumoto}). On the one hand the representation theory of affine Hecke algebras is well behaved (for example, the irreducible representations of these infinite dimensional algebras are all finite dimensional), while on the other hand the representation theory is rather delicate (for instance see the remarkable geometric classification of the irreducibles given in \cite{KL} using the $K$-theory of the flag variety). 

Affine Hecke algebras have a basis $\{T_w\mid w\in W\}$ indexed by elements of an \textit{affine Weyl group}~$W$, and depend on parameters $q_0,\ldots,q_n$ (one parameter for each Coxeter generator $s_0,\ldots,s_n$ of~$W$). The most studied case is the \textit{$1$-parameter case}, where $q_i=q$ for all $i$. It is to this case that the geometric classification mentioned above applies. In \cite{ramcalib} Ram introduced an explicit combinatorial construction of the class of \textit{calibrated representations} of $1$-parameter affine Hecke algebras. These are the modules which have a basis of simultaneous eigenvectors for all the elements of a natural large commutative subalgebra of the Hecke algebra.
While not all representations of an affine Hecke algebra are calibrated, the calibrated representations are of particular interest to combinatorialists since they are the generalisation of the classical combinatorial constructions for Weyl groups to (one parameter) Hecke algebras. Our first aim in this paper is to extend the construction of calibrated representations to the multi-parameter case (see Theorem~\ref{thm:repmain}). We suspect that a full classification of calibrated representations, along the lines of the one parameter case, is possible, although we defer this investigation to later work and instead the focus here is on constructing calibrated representations. In Section~\ref{sect:examples} we give some explicit examples of our calibrated representations, and in Section~\ref{sect:characters} we develop the character theory of calibrated representations in preparation for the Plancherel Theorems.

In the second part of this paper we derive the Plancherel Theorem for rank~$1$ and~$2$ affine Hecke algebras. The Plancherel Theorem is the spectral decomposition of the canonical trace functional $\Tr:\scH\to\CC$ with
$
\Tr(T_w)=\delta_{w,1}$ for $w$ in the affine Weyl group~$W$.
It is the analogue of the formula
$$
\Tr(a)=\sum_{\pi\in\mathrm{Irrep}(\scH)}m_{\pi}\chi_{\pi}(a)
$$
for finite dimensional Hecke algebras, where $m_{\pi}$ are the \textit{generic degrees} (see \cite[Chapter~11]{geck}). For affine Hecke algebras the sum becomes an integral over representations of a $C^*$-algebra completion of $\scH$, and the weights $m_{\pi}$ become the \textit{Plancherel measure}.

 The Plancherel Theorem has been proven in general by Heckman and Opdam \cite{heckmanopdam} and Opdam~\cite{opdamhanalysis} in a veritable tour-de-force paralleling Harish-Chandra's work~\cite{chandra} on the Plancherel Theorem for real and $p$-adic Lie groups (see also Reeder \cite{reeder}). The Plancherel Theorem has been further developed by Delorme-Opdam, Opdam, Opdam-Solleveld, and Ciubotaru-Kato-Kato (\cite{delorme}, \cite{opdam2}, \cite{opdam3}, \cite{Ciubotaru}). Therefore we should explain the value of our direct calculations in ranks~$1$ and~$2$.

Firstly, while the general formulation of the Plancherel Theorem in~\cite{opdamhanalysis} is essentially complete, there are some constants that are not explicitly computed (they are conjectured in \cite[Conjecture~2.27]{opdamhanalysis} to be rational numbers). Thus it is desirable to have a complete and direct calculation in ranks~$1$ and~$2$ which evaluate all constants involved. (We note that in the case of the affine Hecke algebra of the general linear group over a non-archimedean local field, the Plancherel Formula if entirely known, see \cite[Remark~5.6]{AP:05}). 

Secondly, for concrete applications of the Plancherel Theorem (for example, probabilistic calculations like in~\cite{PS}) one may need explicit constructions of the representations involved in the Plancherel formula. For the non-expert this may be a difficult task to fulfill, and so we believe that the combination of both parts of this paper, with a very concrete matching up of representations and terms in the Plancherel Theorem, is of value. In particular the use of calibrated representations makes the Plancherel Theorem accessible at a combinatorial level. 

Finally we hope that the explicit calculations may in some ways serve as an introduction to the general theory, and illustrate the complexity involved in the sophisticated work~\cite{opdamhanalysis}, \cite{Ciubotaru}. The starting point and general philosophy for our derivation of the Plancherel Theorems is similar to that in \cite{opdamhanalysis}, but since we restrict to the rank~$1$ and~$2$ cases the calculations can be carried out by hand. In fact our calculations form an extension of Matsumoto's influential rank~$1$ calculations~\cite[\S2.6]{matsumoto}, and hence provides a companion piece to~\cite{matsumoto} (see also Kutzko and Morris~\cite{KM:09}).

\section{Definitions and setup}

\subsection{Root systems and Weyl groups}

Let $R$ be an irreducible (not necessarily reduced) finite crystallographic root system with simple roots $\alpha_1,\ldots,\alpha_n$ in an $n$-dimensional real vector space $V$ with inner product $\langle\cdot,\cdot\rangle$. Let $R^+$ be the set of positive roots relative to the simple roots $\alpha_1,\ldots,\alpha_n$. Let $W_0$ be the \textit{Weyl group}; the subgroup of $GL(V)$ generated by the reflections $s_{\alpha}$, $\alpha\in R$, where
$
s_{\alpha}(\lambda)=\lambda-\langle\lambda,\alpha\rangle\alpha^{\vee}$ with $\alpha^{\vee}=2\alpha/\langle\alpha,\alpha\rangle$.
Thus $W_0$ is a Coxeter group with distinguished generators $s_1,\ldots,s_n$ (where $s_i=s_{\alpha_i}$). Let $w_0$ be the (unique) longest element of $W_0$. The \textit{dual root system} is $R^{\vee}=\{\alpha^{\vee}\mid \alpha\in R\}$. The \textit{coroot lattice} $Q$ and the \textit{coweight lattice} $P$ are
$$
Q=\ZZ\textrm{-span of }R^{\vee}\qquad\textrm{and}\qquad P=\ZZ\omega_1\oplus\cdots\oplus\ZZ\omega_n,
$$
where $\omega_1,\ldots,\omega_n$ are the \textit{fundamental coweights} defined by $\langle\alpha_i,\omega_j\rangle=\delta_{ij}$. The cone of \textit{dominant coweights} is $P^+=\ZZ_{\geq0}\omega_1\oplus\cdots\oplus\ZZ_{\geq 0}\omega_n$. Then $Q\subseteq P$, and $W_0$ acts on lattices $L$ with $Q\subseteq L\subseteq P$. The \textit{affine Weyl group} associated to $R$ and $L$ is
$$
W_L=L\rtimes W_0,
$$
where we identify $\lambda\in L$ with the translation $t_{\lambda}(x)=x+\lambda$.
Let $\varphi$ be the highest root of $R$, and let $s_0=t_{\varphi^{\vee}}s_{\varphi}$. Let $S=\{s_0,\ldots,s_n\}$. Then $W_Q=\langle S\rangle$ is a Coxeter group, and 
$$
W_L=W_Q\rtimes(L/Q),\qquad\textrm{where $L/Q$ is finite and abelian}.
$$
The \textit{length} $\ell(w)$ of $w\in W_Q$ is the minimum $\ell\geq 0$ such that $w$ can be written as a product of $\ell$ generators in $S$. The length of $w\in W_L$ is defined by $\ell(w)=\ell(w')$, where $w=w'\gamma$ with $w'\in W_Q$ and $\gamma\in L/Q$. Thus elements of $L/Q$ have length zero.

Write $R=R_1\cup R_2\cup R_3$ with 
\begin{align*}
R_1&=\{\alpha\in R\mid \alpha/2\notin R\textrm{ and }2\alpha\notin R\},&R_2&=\{\alpha\in R\mid 2\alpha\in R\},&R_3&=\{\alpha\in R\mid \alpha/2\in R\}.
\end{align*}
These sets are pairwise disjoint, and if $R$ is reduced then $R_1=R$ and $R_{2}=R_3=\emptyset$. Define
$$
R_0=R_1\cup R_2.
$$

The \textit{inversion set} of $w\in W$ is $R(w)=\{\alpha\in R_0^+\mid w^{-1}\alpha\in -R_0^+\}$. By \cite[VI, \S1]{bourbaki} we have
\begin{align}\label{eq:inversionset}
R(w)=\{\alpha_{i_1},s_{i_1}\alpha_{i_2},\ldots,s_{i_1}\cdots s_{i_{\ell-1}}\alpha_{i_{\ell}}\}\quad\textrm{whenever $w=s_{i_1}\cdots s_{i_{\ell}}$ is reduced.}
\end{align}

For each rank $n\geq 1$ there is exactly one irreducible non-reduced root system (up to isomorphism). This is the $BC_n$ system, and it can be realised in $\mathbb{R}^n$ by 
$$
R=\pm\{e_i-e_j,e_i+e_j,e_k,2e_k\mid 1\leq i<j\leq n\textrm{ and }1\leq k\leq n\},
$$
where the simple roots are $\alpha_i=e_i-e_{i+1}$ for $1\leq i<n$ and $\alpha_n=e_n$. The coroot lattice is spanned by $\alpha_1^{\vee},\ldots,\alpha_{n-1}^{\vee},\alpha_n^{\vee}/2$, and we have $P=Q$. Then $R_0$ is a root system of type $B_n$ with simple roots $\alpha_1,\ldots,\alpha_n$. We will always use the above conventions for indexing the simple roots of $BC_n$ root systems, and more generally we will adopt standard Bourbaki conventions~\cite{bourbaki} for the irreducible root systems.

\subsection{Parameter systems}

Let $q_0,q_1,\ldots,q_n\in\mathbb{C}^{\times}$ be such that $q_i=q_j$ whenever $s_i$ and $s_j$ are conjugate in $W_Q$. We call the sequence $(q_i)$ a \textit{parameter system}. By \cite[IV, \S5, No.5, Prop~5]{bourbaki} the product 
$$
q_w=q_{i_1}\cdots q_{i_{\ell}}\quad\textrm{(where $w=s_{i_1}\cdots s_{i_{\ell}}\in W_Q$ is reduced)}
$$
does not depend on the particular reduced expression for~$w$.
If $\alpha\in W_0\alpha_i\cap W_0\alpha_j$ then $s_i$ and $s_j$ are conjugate in $W_0$, and hence for $\alpha\in R_0$ we define
$$
q_{\alpha}=q_i\qquad\textrm{if $\alpha\in W_0 \alpha_i$}.
$$

Let $\CC[L]=\CC\textrm{-span }\{x^{\lambda}\mid \lambda\in L\}$ be the group algebra of $L$, with the group operation written multiplicatively as $x^{\lambda}x^{\mu}=x^{\lambda+\mu}$. In the field of fractions of $\CC[L]$ let (for $\alpha\in R_0$)
\begin{align*}
c_{\alpha}(x)=\begin{cases}\displaystyle{\frac{1-q_{\alpha}^{-1}x^{-\alpha^{\vee}}}{1-x^{-\alpha^{\vee}}}}&\textrm{if $\alpha\in R_1$}\\
\displaystyle{\frac{(1-q_0^{-1/2}q_n^{-1/2}x^{-\alpha^{\vee}/2})(1+q_0^{1/2}q_n^{-1/2}x^{-\alpha^{\vee}/2})}{1-x^{-\alpha^{\vee}}}}&\textrm{if $\alpha\in R_2$.}
\end{cases}
\end{align*}
(Note that if $\alpha\in R_2$ then $2\alpha\in R$, and so $(2\alpha)^{\vee}=\alpha^{\vee}/2$ is in~$L$.) Choose relatively prime elements $n_{\alpha}(x)$ and $d_{\alpha}(x)$ in $\CC[L]$ so that
$$
c_{\alpha}(x)=\frac{n_{\alpha}(x)}{d_{\alpha}(x)}.
$$
For example, $\alpha\in R_2$ and $q_n^{1/2}=q_0^{1/2}$ then $n_{\alpha}(x)=1-q_n^{-1}x^{-\alpha^{\vee}/2}$ and $d_{\alpha}(x)=1-x^{-\alpha^{\vee}/2}$. 

Similarly, let
\begin{align*}
c_{\alpha}'(x)=\begin{cases}
\displaystyle{\frac{1-q_{\alpha}^{-1}}{1-x^{-\alpha^{\vee}}}}&\textrm{if $\alpha\in R_1$}\\
\displaystyle{\frac{1-q_n^{-1}+q_n^{-1/2}(q_0^{1/2}-q_0^{-1/2})x^{-\alpha^{\vee}/2}}{1-x^{-\alpha^{\vee}}}}&\textrm{if $\alpha\in R_2$}.
\end{cases}
\end{align*}
Then, with $d_{\alpha}(x)$ as above,
$$
c_{\alpha}'(x)=\frac{n_{\alpha}'(x)}{d_{\alpha}(x)}
$$
for some $n_{\alpha}'(x)\in\CC[L]$ with $n_{\alpha}'(x)$ and $d_{\alpha}(x)$ relatively prime. 

Let
\begin{align*}
c(x)&=\prod_{\alpha\in R_0^+}c_{\alpha}(x),&n(x)&=\prod_{\alpha\in R_0^+}n_{\alpha}(x),&d(x)&=\prod_{\alpha\in R_0^+}d_{\alpha}(x).
\end{align*}
The expression $c(x)=n(x)/d(x)$ is the \textit{Macdonald $c$-function}. We write $c_i(x)$, $c_i'(x)$, $n_i(x)$, $n_i'(x)$, and $d_i(x)$ for $n_i(x)$ for $c_{\alpha_i}(x)$, $c_{\alpha_i}'(x)$, $n_{\alpha_i}(x)$, $n_{\alpha_i}'(x)$, and $d_{\alpha_i}(x)$ (respectively). 

\subsection{Affine Hecke algebras}

Standard references for affine Hecke algebras include \cite{lusztig}, \cite{macblue} and \cite{ram1}. With the above notation,
the \textit{affine Hecke algebra} $\scH_L$ with \textit{parameters} $q_0,\ldots,q_n$ is the algebra over $\CC$ with vector space basis $\{T_w\mid w\in W_L\}$ and relations
\begin{align*}
T_uT_v&=T_{uv}&&\textrm{if $\ell(uv)=\ell(u)+\ell(v)$}\\
T_{i}^2&=1+(q_i^{\frac{1}{2}}-q_i^{-\frac{1}{2}})T_i&&\textrm{for all $i=0,1,\ldots,n$},
\end{align*}
where we write $T_i=T_{s_i}$.

The above presentation is the \textit{Coxeter} presentation of $\scH_L$. There is a second important presentation which exploits the semidirect product structure $W_L=L\rtimes W_0$. This is 
the \textit{Bernstein presentation}, given by (\ref{eq:rel1})--(\ref{eq:rel4}) below. Each $\lambda\in L$ can be written as $\lambda=\mu-\nu$ with $\mu,\nu\in L\cap P^+$, and we define
$$
x^{\lambda}=T_{t_{\mu}}T_{t_{\nu}}^{-1}.
$$
It is not hard to see that this is well defined, and in particular $x^{\lambda}=T_{t_{\lambda}}$ if $\lambda$ is dominant. 

It can be shown \cite{macblue} that $\scH_L$ has vector space basis
$\{T_wx^{\lambda}\mid \lambda\in L,w\in W_0\}$ 
and relations
\begin{align}
\label{eq:rel1}T_i^2&=1+(q_i^{\frac{1}{2}}-q_i^{-\frac{1}{2}})T_i&&\textrm{for $i=1,\ldots,n$}\\
\label{eq:rel2}T_iT_jT_i\cdots&=T_jT_iT_j\cdots\quad\textrm{($m_{ij}$ factors)}&&\textrm{for $1\leq i<j\leq n$}\\
\label{eq:rel3}x^{\lambda}x^{\mu}&=x^{\lambda+\mu}&&\textrm{for all $\lambda,\mu\in L$}\\
\label{eq:rel4}T_ix^{\lambda}-x^{s_i\lambda}T_i&=q_i^{\frac{1}{2}}c_i'(x)(x^{\lambda}-x^{s_i\lambda})&&\textrm{for $1\leq i\leq n$ and $\lambda\in L$.}
\end{align}
Thus we see a copy of the group algebra
$\CC[L]$ of the lattice $L$ inside of $\scH_L$. The relation (\ref{eq:rel4}) is the \textit{Bernstein relation}. Since $s_i\lambda=\lambda-\langle \lambda,\alpha_i\rangle\alpha_i^{\vee}$ and $\langle\lambda,\alpha_i\rangle\in\ZZ$ the ``fraction'' $c_i'(x)(x^{\lambda}-x^{s_i\lambda})$ that appears in the Bernstein relation is actually an element of $\CC[L]$.

It is well known (see, for example, \cite[(4.2.10)]{macblue}) that the centre of $\scH_L$ is
$$
\CC[L]^{W_0}=\{f\in\CC[L]\mid w\cdot f=f\textrm{ for all $w\in W_0$}\}.
$$
This has powerful implications for the representation theory of $\scH_L$. For example it forces the irreducible representations to be finite dimensional (since the centre acts on irreducible representations by scalars, and it is evident that $\scH_L$ is finite dimensional over $\CC[L]^{W_0}$).

It is natural to seek modifications $\tau_w$ of $T_w$ which satisfy a ``simplified Bernstein relation'' 
\begin{align}
\label{eq:simpbernstein}
\tau_w x^{\lambda}=x^{w\lambda}\tau_w\qquad\textrm{for all $w\in W_0$ and $\lambda\in L$}.
\end{align}
Define elements $\tau_1,\ldots,\tau_n\in\scH_L$ by
$$
\tau_i=d_i(x)T_i-q_i^{\frac{1}{2}}n_i'(x).
$$
The Bernstein relation gives $\tau_ix^{\lambda}=x^{s_i\lambda}\tau_i$, and it can be shown (see \cite[Proposition~2.7]{ramcalib} for example) that for $w\in W_0$ the product
$$
\tau_w=\tau_{i_1}\cdots \tau_{i_{\ell}}\quad\textrm{is independent of the choice of reduced expression $w=s_{i_1}\cdots s_{i_{\ell}}$.}
$$
Thus the elements $\tau_w$, $w\in W_0$, satisfy (\ref{eq:simpbernstein}), and a direct calculation gives the useful formula
\begin{align}\label{eq:tausquared}
\tau_i^2=q_i\,n_i(x)n_i(x^{-1})\in\CC[L].
\end{align}

\subsection{Harmonic analysis for the affine Hecke algebra}

Suppose now that $q_0,q_1,\ldots,q_n>1$. Define an involution $*$ on $\scH_L$ and the \textit{canonical trace functional} $\Tr:\scH_L\to\CC$ by
\begin{align*}
\bigg(\sum_{w\in W_L}c_wT_w\bigg)^*=\sum_{w\in W_L}\overline{c_w}T_{w^{-1}}\quad\textrm{and}\quad \Tr\bigg(\sum_{w\in W_L}c_wT_w\bigg)=c_1.
\end{align*}
An induction on $\ell(v)$ shows that $\Tr(T_uT_v^*)=\delta_{u,v}$ for all $u,v\in W_L$, and so
\begin{align*}
\Tr(h_1h_2)=\Tr(h_2h_1)\qquad\textrm{for all $h_1,h_2\in\scH_L$}.
\end{align*}
Thus 
$
(h_1,h_2)=\Tr(h_1h_2^*)
$
defines an Hermitian inner product on $\scH_L$. Let $\|h\|_2=\sqrt{(h,h)}$. The algebra $\scH_L$ acts on itself by left multiplication, and the corresponding operator norm is
$$
\|h\|=\sup\{\|hx\|_2\,:\,x\in\scH_L,\|x\|_2\leq 1\}.
$$
Let $\overline{\scH_L}$ denote the completion of $\scH_L$ with respect to this norm. Thus $\overline{\scH_L}$ is a non-commutative $C^*$-algebra. This algebra is `liminal'. Even better, all irreducible representations of $\overline{\scH_L}$ are finite dimensional, and so by \cite[\S8.8]{dixmier} there exists a probability measure $\mu$ such that
\begin{align}
\label{eq:plancherelgeneral}\Tr(h)=\int_{\mathrm{spec}(\overline{\scH_L})}\chi_{\pi}(h)d\mu(\pi)\qquad\textrm{for all $h\in\overline{\scH_L}$.}
\end{align}
 The measure $\mu$ is the \textit{Plancherel measure}. Only those representations of $\scH_L$ which extend to the completion $\overline{\scH_L}$ appear in the Plancherel Theorem. It is known \cite[Corollary~6.2]{opdamhanalysis} that these are the \textit{tempered} representations of $\scH_L$ (see \cite[\S2.7]{opdamhanalysis} for the definition).

If $t\in\Hom(L,\CC^{\times})$ we write $t^{\lambda}=t(\lambda)$. The Weyl group $W_0$ acts on $\Hom(L,\CC^{\times})$ by the formula $(wt)^{\lambda}=t^{w^{-1}\lambda}$. Following \cite{opdamtrace}, define a function $G_t:\scH_L\to\CC$ by
\begin{align}\label{eq:Gdefn}
G_t(h)=\sum_{\mu\in L}t^{-\mu}\Tr(x^{\mu} h)
\end{align}
whenever the series converges. From \cite{opdamtrace} we have the following (see also \cite[(3.9)]{opdamhanalysis}).

\begin{thm}\label{thm:prelim_main} The series $G_t(h)$ is absolutely convergent for all $h\in\scH_L$ whenever $|t^{\alpha_i^{\vee}}|<q_i^{-1}$ for $(R,i)\neq (BC_n,n)$ and $|t^{\alpha_n^{\vee}}|<q_0^{-1}q_n^{-1}$ for $(R,i)=(BC_n,n)$. Moreover,
\begin{align}\label{eq:prelim_main1}
G_t(h)=\frac{g_t(h)}{q_{w_0}c(t)c(t^{-1})d(t)}
\end{align}
where for each fixed $h$ the function $g_t(h)$ has a analytic continuation in the $t$-variable to $\Hom(L,\CC^{\times})$. Moreover, $g_t(h)$ satisfies
\begin{enumerate}
\item $g_t(h)$ is a polynomial in $\{t^{\lambda}\mid \lambda\in L\}$ (for fixed $h\in\scH_L$),
\item $g_t(1)=d(t)$ for all $t\in\Hom(L,\CC^{\times})$, and
\item $g_t(x^{\lambda}hx^{\mu})=t^{\lambda+\mu}g_t(h)$ for all $\lambda,\mu\in L$ and all $h\in\scH_L$.
\end{enumerate}
\end{thm}

\begin{remark}\label{rem:explicit}
(a) Note that $t^{\lambda}g_t(\tau_w)=g_t(\tau_w x^{\lambda})=g_t(x^{w\lambda}\tau_w)=t^{w\lambda}g_t(\tau_w),
$
and so if $wt\neq t$ then
$$g_t(\tau_wx^{\lambda})=\delta_{w,1}t^{\lambda}d(t). 
$$
Then by condition 1 in the theorem this formula holds for all $t\in\Hom(L,\CC^{\times})$.

(b) The three conditions in the theorem completely determine $g_t(h)$. For example consider the $\tilde{A}_2$ case. It is sufficient to calculate $g_t(T_w)$ for each $w\in W_0$, because $g_t(T_wx^{\lambda})=t^{\lambda}g_t(T_w)$. Write $Q=q^{\frac{1}{2}}-q^{-\frac{1}{2}}$. Applying $g_t$ to the Bernstein relation $T_1x^{\alpha_1^{\vee}}=x^{-\alpha_1^{\vee}}T_1+Q(1+x^{\alpha_1^{\vee}})$ gives
$$
g_t(T_1)=Q(1-t^{-\alpha_1^{\vee}})^{-1}d(t)=Q(1-t^{-\alpha_2^{\vee}})(1-t^{-\alpha_1^{\vee}-\alpha_2^{\vee}}).
$$
Similarly, $g_t(T_2)=Q(1-t^{-\alpha_1^{\vee}})(1-t^{-\alpha_1^{\vee}-\alpha_2^{\vee}})$, $g_t(T_1T_2)=g_t(T_2T_1)=Q^2(1-t^{-\alpha_1^{\vee}-\alpha_2^{\vee}})$, and $g_t(T_1T_2T_1)=Q^3+Q(1-t^{-\alpha_1^{\vee}})(1-t^{-\alpha_2^{\vee}})$, making (\ref{eq:prelim_main1}) completely explicit in type~$\tilde{A}_2$.
\end{remark}

Let
\begin{align}\label{eq:f}
f_t(h)=\frac{g_t(h)}{d(t)}.
\end{align}
Note that $f_t(h)$ may have poles at points where $t^{\alpha^{\vee}}=1$ for some $\alpha\in R_0$. Fix a $\ZZ$-basis $\lambda_1,\ldots,\lambda_n$ of $L$. From (\ref{eq:Gdefn}) and (\ref{eq:prelim_main1}) we have
\begin{align}\label{eq:prelim_main}
\Tr(h)=\frac{1}{q_{w_0}}\int_{a_1\TT}\cdots\int_{a_n\TT} \frac{f_t(h)}{c(t)c(t^{-1})}\,dt_1\cdots dt_n
\end{align}
where $t_i=t^{\lambda_i}$, $dt_i$ is Haar measure on the circle group $\TT$, and where $a_1,\ldots,a_n>0$ are chosen such that if $t\in\Hom(L,\CC^{\times})$ with $|t^{\lambda_i}|=a_i$ for each $i$ then 
$|t^{\alpha_i^{\vee}}|<q_i^{-1}$ (if $(R,i)\neq (BC_n,n)$) and $|t^{\alpha_n^{\vee}}|<q_0^{-1}q_n^{-1}$ (if $(R,i)=(BC_n,n)$). Formula (\ref{eq:prelim_main}) appears in \cite[Theorem~3.7]{opdamhanalysis}, and is the starting point of the Plancherel Theorem.

\section{Representations of affine Hecke algebras}

Let $M$ be a finite dimensional $\scH_L$-module. For each $t\in\Hom(L,\CC^{\times})$ let 
\begin{align*}
M_t&=\{v\in M\mid x^{\lambda}\cdot v=t^{\lambda}v\quad\textrm{for all $\lambda\in L$}\}\\
M_t^{\mathrm{gen}}&=\{v\in M\mid \textrm{for each $\lambda\in L$ there is a $k>0$ such that $(x^{\lambda}-t^{\lambda})^k\cdot v=0$}\}
\end{align*}
be the \textit{$t$-weight space} and the \textit{generalised $t$-weight space} respectively. Each finite dimensional $\scH_L$-module $M$ decomposes into a direct sum of generalised $t$-weight spaces
$$
M=\bigoplus_{t\in\mathrm{supp}(M)}M_t^{\mathrm{gen}}
$$
where $\mathrm{supp}(M)=\{t\in\Hom(L,\CC^{\times})\mid M_t^{\mathrm{gen}}\neq 0\}$ is the \textit{support} of $M$.

A finite dimensional $\scH_L$-module $M$ is \textit{calibrated} if 
$$
M_t^{\mathrm{gen}}=M_t\qquad\textrm{for all $t\in\mathrm{supp}(M)$}.
$$
(In the literature this is also refereed to as \textit{tame}). The main purpose of the first half of this paper is to construct calibrated irreducible representations of general affine Hecke algebras. We suspect that a complete classification of calibrated representations along the lines of the $1$-parameter case is possible (perhaps with some restrictions like $W_t=W_{(t)}$), although such a classification requires detailed information about the representation theory of rank~$2$ (multi-parameter) affine Hecke algebras, and this would take us beyond the scope of the present paper. Thus the focus of this paper is on construction, and the question of classification will be pursued in later work.

The elements $\tau_i\in\scH_L$, considered as operators on a representation, are often called \textit{intertwining operators} because of the following fundamental and important fact.

\begin{lemma}\label{lem:bijection} Let $1\leq i\leq n$. Let $M$ be a finite dimensional $\scH_L$-module, and suppose that $t\in\mathrm{supp}(M)$. Then as operators,
$$
\tau_i:M^{\mathrm{gen}}_t\to M^{\mathrm{gen}}_{s_it}\qquad\textrm{and}\qquad\tau_i:M_{s_it}^{\mathrm{gen}}\to M_t^{\mathrm{gen}}.
$$
Moreover, $n_i(t)n_i(t^{-1})\neq 0$ if and only if each operator is bijective. 
\end{lemma}

\begin{proof}
Let $m\in M_t^{\mathrm{gen}}$. By (\ref{eq:simpbernstein}) we compute 
$$
(x^{\lambda}-(s_it)^{\lambda})^k\tau_i\cdot m=\tau_i(x^{s_i\lambda}-t^{s_i\lambda})^k\cdot m=0
$$
for sufficiently large~$k$. Thus $\tau_i\cdot m\in M_{s_it}^{\mathrm{gen}}$. For the final claim, note that by~(\ref{eq:tausquared}) the operator $\tau_i^2:M_t^{\mathrm{gen}}\to M_t^{\mathrm{gen}}$ is given by $\tau_i^2\cdot m=q_in_i(t)n_i(t^{-1})m$. 
\end{proof}

By Schur's Lemma (see~\cite{varadarajan}) the centre $\CC[L]^{W_0}$ of $\scH_L$ acts on an irreducible module $M$ by scalars. It follows that there exists $t\in \Hom(L,\CC^{\times})$ such that
$$
p(x)\cdot v=p(t)v\qquad\textrm{for all $p(x)\in\CC[L]^{W_0}$ and all $v\in M$.}
$$
The element $t$ is only defined up to $W_0$ orbits. The orbit $W_0t$ is called the \textit{central character} of~$M$, although as is customary we will usually refer to any $t'\in W_0t$ as `the' central character of~$M$. A central character $t\in\Hom(L,\CC^{\times})$ is called \textit{regular} if $t^{\alpha^{\vee}}\neq 1$ for all $\alpha\in R_0$.

\subsection{Principal series representations}

The large commutative subalgebra $\CC[L]$ of $\scH_L$ can be used to construct finite dimensional representations of~$\scH_L$. The \textit{principal series representation} with \textit{central character} $t\in\Hom(L,\CC^{\times})$ is
$$
M(t)=\mathrm{Ind}_{\CC[L]}^{\scH_L}(\CC v_t)=\scH_L\otimes_{\CC[L]}(\CC v_t),
$$
where $\CC v_t$ is the $1$-dimensional representation of $\CC[L]$ with $x^{\lambda}\cdot v_t=t^{\lambda} v_t$ for all $\lambda\in L$. It is clear that this representation is $|W_0|$-dimensional, and that $\{T_w\otimes v_t\mid w\in W_0\}$ is a basis of $M(t)$.

For $t\in\Hom(L,\CC^{\times})$ define
\begin{align*}
N(t)&=\{\alpha\in R_0^+\mid n_{\alpha}(t)n_{-\alpha}(t)=0\}\\
D(t)&=\{\alpha\in R_0^+\mid d_{\alpha}(t)=0\}.
\end{align*}
Note that $N(t)$ and $D(t)$ are closely related to the zeros of the numerator and denominator of the Macdonald $c$-function (respectively). 

For $t\in\Hom(L,\CC^{\times})$, let
\begin{align*}
W_t&=\{w\in W_0\mid wt=t\}\\
W_{(t)}&=\langle\{s_{\alpha}\mid \alpha\in D(t)\}\rangle.
\end{align*}
Note that $W_{(t)}$ is a normal subgroup of $W_t$ (since $ws_{\alpha}w^{-1}=s_{w\alpha}$). Moreover, if $L=P$ then necessarily $W_{(t)}=W_t$ (see \cite[\S4.2, 5.3]{steinberg2}).

The following Theorem of Kato \cite[Theorem~2.2]{kato} is fundamental.

\begin{thm}\label{thm:kato} The module $M(t)$ is irreducible if and only if 
$N(t)=\emptyset$ and $W_t=W_{(t)}$.
\end{thm}

The fundamental importance of the principal series representations is highlighted by the following fact (see, for example, \cite[Proposition~2.6]{ramcalib}).

\begin{prop}\label{prop:quotient}
If $M$ is an irreducible representation of $\scH_L$ with central character $t$ then $M$ is a quotient of $M(t)$. In particular $\dim(M)\leq |W_0|$.
\end{prop}

\subsection{A combinatorial construction of irreducible calibrated $\scH_L$-modules}\label{sect:combconst}

Following \cite{ramcalib}, the \textit{calibration graph} of $t\in\Hom(L,\CC^{\times})$ is the graph $\Gamma(t)$ with 
\begin{align*}
&\textrm{vertex set $\{wt\mid w\in W_0\}$, and}\\
&\textrm{edges $\{wt,s_iwt\}$ if and only if $\alpha_i\notin N(wt)$.}
\end{align*}
For each $J\subseteq N(t)$ define
$$
F_J(t)=\{w\in W_0\mid R(w^{-1})\cap N(t)=J\textrm{ and }R(w^{-1})\cap D(t)=\emptyset\}.
$$
By the argument in~\cite[Theorem~2.14]{ramcalib}, if $W_t=W_{(t)}$ then the connected components of $\Gamma(t)$ are precisely the sets
\begin{align}\label{eq:concal}
\{wt\mid w\in F_J(t)\}\qquad\textrm{such that $J\subseteq N(t)$ and $F_J(t)\neq \emptyset$}.
\end{align}

\begin{remark}\label{rem:calib} We note that if $W_t=W_{(t)}$ then the geometric argument in \cite[Theorem~2.14]{ramcalib} also shows that if $w,v\in F_J(t)$, and if $wv^{-1}=s_{i_1}\cdots s_{i_{\ell}}$ is a reduced expression, then each element
$$
w,\,\,s_{i_1}w,\,\,s_{i_2}s_{i_1}w,\,\,\ldots,\,\,s_{i_{\ell}}\cdots s_{i_2}s_{i_1}w=v\quad\textrm{is in $F_J(t)$}.
$$
(Because the ``smaller regions'' in the proof of \cite[Theorem~2.14]{ramcalib} which correspond to the connected components of the calibration graph are convex in the sense that they are intersections of half spaces, and hence by~\cite[Proposition~2.8]{ronan} all minimal length paths between $w$ and $v$ are contained in this region). Thus $F_J(t)$ is `geodesically closed' in the (dual of the) underlying Coxeter complex.
\end{remark}

\begin{prop}\label{prop:calib} If $M$ is a finite dimensional $\scH_L$-module then 
$$\dim(M_t^{\mathrm{gen}})=\dim(M_{t'}^{\mathrm{gen}})$$ whenever $t$ and $t'$ are in the same connected component of the calibration graph~$\Gamma(t)$.
\end{prop}

\begin{proof} 
If $\alpha_i\notin N(t)$ then Lemma~\ref{lem:bijection} gives $\dim(M_t^{\mathrm{gen}})=\dim(M_{s_it}^{\mathrm{gen}})$. Hence the result.
\end{proof}

Let $R_{ij}$ be the rank~$2$ subsystem of $R$ generated by the simple roots $\alpha_i$ and $\alpha_j$. That is, $R_{ij}$ is the intersection of $R$ with the $\ZZ$-span of $\{\alpha_i,\alpha_j\}$. We say that a weight $t\in\Hom(L,\CC^{\times})$ is \textit{$(i,j)$-regular} if $(wt)^{\alpha_i^{\vee}}\neq 1$ and $(wt)^{\alpha_j^{\vee}}\neq 1$ for all $w\in W_{ij}=\langle s_i,s_j\rangle$, and $(i,j)$-\textit{calibratable} if one of the following conditions holds:
\begin{enumerate}
\item[(1)]  The weight $t$ is $(i,j)$-regular.
\item[(2)] $R_{ij}$ is of type $C_2$ (assume $\alpha_i$ short and $\alpha_j$ long) with 
\begin{enumerate}
\item[(a)] $q_i=q_j$ and $(t^{\alpha_i^{\vee}},t^{\alpha_j^{\vee}})=(q_i,q_i^{-1})$ or $(q_i^{-1},q_i)$, or
\item[(b)] $q_i=q_j^2$ and $(t^{\alpha_i^{\vee}},t^{\alpha_j^{\vee}})=(q_j^{-2},q_j)$ or $(q_j^2,q_j)$. 
\end{enumerate}
\item[(3)] $R_{ij}$ is of type $G_2$ (assume $\alpha_i$ short and $\alpha_j$ long) with 
\begin{enumerate}
\item[(a)] $q_i=q_j$ and $(t^{\alpha_i^{\vee}},t^{\alpha_2^{\vee}})=(q_i^{-1},q_i),(q_i,q_i^{-1}),(q_i^2,q_i^{-1}),(q_i^{-2},q_i)$, or
\item[(b)] $q_i=q_j^3$ and $(t^{\alpha_i^{\vee}},t^{\alpha_j^{\vee}})=(q_j^3,q_j^{-1}),(q_j^{-3},q_j),(q_j^{-3},q_j^2),(q_j^{3},q_j^{-2})$.
\end{enumerate}
\item[(4)] $R_{ij}$ is of type $BC_2$ (assume $\alpha_i$ middle-length and $\alpha_j$ short) with 
\begin{enumerate}
\item[(a)] $q_i=q_0q_j$ and $(t^{\alpha_i^{\vee}},t^{\alpha_j^{\vee}/2})=(q_0q_j,q_0^{-1/2}q_j^{-1/2}),(q_0^{-1}q_j^{-1},q_0^{1/2}q_j^{1/2})$, or
\item[(b)] $q_i=q_0q_j^{-1}$ or $q_0^{-1}q_j$ and $(t^{\alpha_i^{\vee}},t^{\alpha_j^{\vee}/2})=(q_0^{-1}q_j,-q_0^{1/2}q_j^{-1/2}),(q_0q_j^{-1},-q_0^{-1/2}q_j^{1/2})$, or
\item[(c)] $q_i=q_0^{1/2}q_j^{1/2}$ and $(t^{\alpha_i^{\vee}},t^{\alpha_j^{\vee}/2})=(q_i^{-1},q_i),(q_i,q_i^{-1})$, or
\item[(d)] $q_i=q_0^{-1/2}q_j^{1/2}$ or $q_0^{1/2}q_j^{-1/2}$ and $(t^{\alpha_i^{\vee}},t^{\alpha_j^{\vee}/2})=(q_i^{-1},-q_i),(q_i,-q_i^{-1})$.
\end{enumerate}
\end{enumerate}
Conditions (1), (2)(a) and (3)(a) are equivalent to Ram's definition of calibratable in the $1$-parameter case. Note that if $R_{ij}$ is of type $BC_2$ then the underlying root system~$R$ is necessarily non-reduced, and hence is of type $BC_n$, and so $1\leq i\leq n-1$ and $j=n$ (since $\alpha_j$ is assumed to be short).

In the following theorem we construct a class of irreducible calibrated $\scH_L$-modules.

\begin{thm}\label{thm:repmain} Let $t\in\Hom(L,\CC^{\times})$, and let $J\subseteq N(t)$. Suppose that $F_J(t)\neq\emptyset$ and that each~$wt$ with $w\in F_J(t)$ is $(i,j)$-calibratable for each pair $(\alpha_i,\alpha_j)$ of simple roots of $R$. Let $M_J(t)$ be the vector space over $\CC$ with basis $\{e_{wt}\mid w\in F_J(t)\}$, and define linear operators $\tilde{T}_i$ ($i=1,\ldots,n$), $\tilde{x}^{\lambda}$ ($\lambda\in L$), on $M_J(t)$ by linearly extending the formulae
\begin{align}
\label{eq:diagonal}\tilde{x}^{\lambda}e_{wt}&=(wt)^{\lambda}e_{wt}&&\lambda\in L\\
\label{eq:Taction}\tilde{T}_ie_{wt}&=q_i^{\frac{1}{2}}c_{i}'(wt)e_{wt}+q_i^{\frac{1}{2}}c_{i}(wt)e_{s_iwt}&&1\leq i\leq n,
\end{align}
with the convention that $e_{vt}=0$ if $v\notin F_J(t)$. Then the map $\scH_L\to \mathrm{End}(M_J(t))$ with $T_i\mapsto \tilde{T}_i$ and $x^{\lambda}\mapsto\tilde{x}^{\lambda}$ defines an irreducible calibrated representation of $\scH_L$. 
\end{thm}

\begin{proof} (a) We check that the operators $\tilde{T}_i$ and $\tilde{x}^{\lambda}$ satisfy the Bernstein relation. We have
 \begin{align*}
 (\tilde{T}_i\tilde{x}^{\lambda}-\tilde{x}^{s_i\lambda}\tilde{T}_i) e_{wt}&=((wt)^{\lambda}-\tilde{x}^{s_i\lambda})\tilde{T}_i e_{wt}=((wt)^{\lambda}-x^{s_i\lambda})\big(q_i^{\frac{1}{2}}c_i'(wt)e_{wt}+q_i^{\frac{1}{2}}c_i(wt)e_{s_iwt}\big)\\
 &=((wt)^{\lambda}-(wt)^{s_i\lambda})q_i^{\frac{1}{2}}c_i'(wt)e_{wt}=q_i^{\frac{1}{2}}c_i'(\tilde{x})(\tilde{x}^{\lambda}-\tilde{x}^{s_i\lambda})e_{wt}.
 \end{align*}

(b) We now check that the operators $\tilde{T}_i$ satisfy the quadratic relation $\tilde{T}_i^2=1+(q_i^{\frac{1}{2}}-q_i^{-\frac{1}{2}})\tilde{T}_i$.
\begin{align*}
\tilde{T}_i^2 e_{wt}&=\tilde{T}_i \big(q_i^{\frac{1}{2}}c_i'(wt)e_{wt}+q_i^{\frac{1}{2}}c_i(wt)e_{s_iwt}\big)\\
&=q_i\left(c_i'(wt)^2+c_i(wt)c_i(s_iwt)\right)e_{wt}+q_ic_i(wt)\left(c_i'(wt)+c_i'(s_iwt)\right)e_{s_iwt}\\
&=\big(1+(q_i^{\frac{1}{2}}-q_i^{-\frac{1}{2}})q_i^{\frac{1}{2}}c_i'(wt)\big)e_{wt}+(q_i^{\frac{1}{2}}-q_i^{-\frac{1}{2}})q_i^{\frac{1}{2}}c_i(wt)e_{s_iwt}=(1+(q_i^{\frac{1}{2}}-q_i^{-\frac{1}{2}})\tilde{T}_i) e_{wt}.
\end{align*}

(c) We verify the braid relation $\cdots \tilde{T}_i\tilde{T}_j\tilde{T}_i=\cdots \tilde{T}_j\tilde{T}_i\tilde{T}_j$ ($m_{ij}$ factors). Fix $w\in F_J(t)$. Suppose first that $wt$ is $(i,j)$-regular. Let $v\in W_{ij}$. If $e_{vwt}\neq 0$ then (\ref{eq:Taction}) gives
\begin{align}\label{eq:ind}
\big(\tilde{T}_i-q_i^{\frac{1}{2}}c_i'(vwt)\big) e_{vwt}=q_i^{\frac{1}{2}}c_i(vwt)e_{s_ivwt},
\end{align}
and by Remark~\ref{rem:calib} this formula is also true when $e_{vwt}=0$ and $\ell(s_iv)>\ell(v)$. 

Consider the product (well defined by $(i,j)$-regularity)
\begin{align*}
A_{ij}(wt)&=\cdots(\tilde{T}_i-q_i^{\frac{1}{2}}c_i'(s_js_iwt))(\tilde{T}_j-q_j^{\frac{1}{2}}c_j'(s_iwt))(\tilde{T}_i-q_i^{\frac{1}{2}}c_i'(wt))&&\textrm{($m_{ij}$ factors)}.
\end{align*}
Let $v_0$ be the longest element of $W_{ij}$. Repeatedly using (\ref{eq:ind}) and $c_{\alpha}(vwt)=c_{v^{-1}\alpha}(wt)$ gives
\begin{align*}
A_{ij}(wt) e_{wt}&=q_{v_0}^{\frac{1}{2}}\big[c_{\alpha_i}(wt)c_{s_i\alpha_j}(wt)c_{s_is_j\alpha_i}(wt)c_{s_is_js_i\alpha_j}(wt)\cdots \big]e_{v_0wt}.
\end{align*}
By (\ref{eq:inversionset}) we have $\{\alpha_i,s_i\alpha_j,s_is_j\alpha_i,\ldots\}=\{\alpha_j,s_j\alpha_i,s_js_i\alpha_j,\ldots\}$ and so
$
A_{ji}(wt) e_{wt}=A_{ij}(wt) e_{wt}.
$
Each $v\in W_{ij}\backslash\{v_0\}$ has a unique expression as a product of simple generators, and so for $v<v_0$ we may unambiguously define operators $\tilde{T}_v=\tilde{T}_{i_1}\tilde{T}_{i_2}\cdots\tilde{T}_{i_{\ell}}$ where $v=s_{i_1}s_{i_2}\cdots s_{i_{\ell}}$ is the unique reduced expression for $v\in W_{ij}$. Expanding $A_{ij}(wt)$ and $A_{ji}(wt)$ and using the already verified quadratic relation for $\tilde{T}_i$ and $\tilde{T}_j$ we see that there are rational functions $p_v(wt)$ and $q_v(wt)$ in $wt$ such that
\begin{align*}
A_{ij}(wt)e_{wt}&=\cdots \tilde{T}_i\tilde{T}_j\tilde{T}_i e_{wt}+\sum_{v<v_0}p_v(wt)\tilde{T}_v e_{wt}\\
A_{ji}(wt)e_{wt}&=\cdots \tilde{T}_j\tilde{T}_i\tilde{T}_j e_{wt}+\sum_{v<v_0}q_v(wt)\tilde{T}_v e_{wt},
\end{align*}
One now shows that $p_v(wt)=q_v(wt)$ for all $v<v_0$. This is achieved exactly as in \cite[Proposition~2.7]{ramcalib} by using the action of the $\tau$-operators on principal series representations, and we omit the details. Thus the braid relation, in the $(i,j)$-regular case, holds.

We now verify the braid relation in the case where $wt$ is $(i,j)$-calibratable but not $(i,j)$-regular. Consider the $R_{ij}=C_2$ case with $q_i=q_j$ and $(wt)^{\alpha_i^{\vee}}=q_i$ and $(wt)^{\alpha_j^{\vee}}=q_i^{-1}$. By (\ref{eq:concal}) we have $F_J(t)=\{wt\}$, and so the braid relation is trivially satisfied (as $M_J(t)$ is $1$-dimensional). All other $C_2$ cases are similar.
In the $G_2$ case with $q_i=q_j^3$ and $(wt)^{\alpha_i^{\vee}}=q_j^3$ and $(wt)^{\alpha_j^{\vee}}=q_j^{-2}$, by (\ref{eq:concal}) we compute $F_J(t)=\{w,s_jw\}$, and a direct calculation gives 
\begin{align*}
\tilde{T}_ie_{wt}&=q_i^{\frac{1}{2}}e_{wt}&\tilde{T}_je_{wt}&=\frac{1}{q_j+1}\left(-q_j^{-\frac{1}{2}}e_{wt}+q_j^{\frac{1}{2}}e_{s_jwt}\right)\\
\tilde{T}_ie_{s_jwt}&=-q_i^{-\frac{1}{2}}e_{s_jwt}&\tilde{T}_je_{s_jwt}&=q_j^{\frac{1}{2}}\left(\frac{1-q_j^{-3}}{1-q_j^{-2}}e_{wt}+\frac{1}{1+q_j^{-1}}e_{s_jwt}\right)
\end{align*}
The braid relation follows by direct calculation. The remaining $G_2$ cases are similar (or trivial). Finally, in all $BC_2$ cases we have $F_J(t)=\{wt\}$ and so the braid relation is trivially satisfied. 

To conclude the proof we show that the module $M_J(t)$ is irreducible and calibrated. By the construction the generalised weight spaces of $M_J(t)$ are $M_J(t)_{wt}$, with $w\in F_J(t)$, and each generalised weight space has dimension~$1$. So $M_J(t)$ is calibrated. Furthermore, it follows that if $M$ is a proper submodule of $M_J(t)$ then there is $w,w'\in F_J(t)$ with $wt\neq w't$ such that $M_{wt}\neq 0$ and $M_{w't}=0$, contradicting Proposition~\ref{prop:calib}. Thus $M_J(t)$ is irreducible.
\end{proof}

We note the following subtle point: In Theorem~\ref{thm:repmain} the basis of $M_J(t)$ is indexed by the set $\{wt\mid w\in F_J(t)\}$, while in the construction \cite[Theorem~3.5]{ram5} the basis is indexed by~$F_J(t)$. The reason for this refinement is that we work with general lattices $Q\subseteq L\subseteq P$, while in \cite{ramcalib,ram5} the lattice $L=P$ is specified. See Examples~\ref{ex:1} and~\ref{ex:2} below.

\begin{remark} Recently \cite{davis} Ram's construction has been applied to study the representation theory of $1$-parameter rank $2$ affine Hecke algebras with $q$ a root of the Poincar\'{e} polynomial, and analogously the above construction could be applied to the study of such representations in the multi-parameter case.
\end{remark}

\subsection{Examples}\label{sect:examples}

Let us give some concrete examples of the construction from Theorem~\ref{thm:repmain}. Most of these examples will arise in the Plancherel Theorems in the later parts of this paper (see Section~\ref{sect:plancherel}). Of interest, we see in the third and fourth examples that some non-calibrated modules (in the single parameter case) can be constructed from calibrated modules (of multi-parameter algebras) by making an appropriate change of basis and taking a limit.

\begin{Example}\label{ex:1}
Let $\scH$ be a $\tilde{C}_2$ Hecke algebra with $L=P$ and with parameters $q_1$ and~$q_2$ (see Section~\ref{sect:C2P}). Let~$t\in\Hom(P,\CC^{\times})$ be the character with $t^{\omega_1}=-q_1^{-1}$ and $t^{\omega_2}=q_1^{-1/2}$, so that $t^{\alpha_1^{\vee}}=q_1^{-1}$ and $t^{\alpha_2^{\vee}}=-1$. Thus $N(t)^{\vee}=\{\alpha_1^{\vee},\alpha_1^{\vee}+2\alpha_2^{\vee}\}$ and $D(t)=\emptyset$. Thus there are~$4$ choices for subsets $J\subseteq N(t)$. Let $J_1^{\vee}=\emptyset$, $J_2^{\vee}=\{\alpha_1^{\vee}\}$, $J_3^{\vee}=\{\alpha_1^{\vee}+2\alpha_2^{\vee}\}$, and $J_4^{\vee}=\{\alpha_1^{\vee},\alpha_1^{\vee}+2\alpha_2^{\vee}\}$. We compute $F_{J_1}(t)=\{1,s_2\}$, $F_{J_2}(t)=\{s_1,s_2s_1\}$, $F_{J_3}(t)=\{s_1s_2,s_2s_1s_2\}$, and $F_{J_4}(t)=\{s_1s_2s_1,s_1s_2s_1s_2\}$, and so the calibration graph is as in Figure~\ref{fig:calibex1}(a). Thus by Theorem~\ref{thm:repmain} and Proposition~\ref{prop:quotient} there are $4$ irreducible modules with central character~$t$, each with dimension~$2$. For example, the matrices for the module $M_{J_1}(t)$ with respect to the basis $\{e_t,e_{s_2t}\}$ are $\pi(T_1)=-q_1^{-1/2}I$, $\pi(x^{\omega_1})=-q_1^{-1}I$, and
\begin{align*}
\pi(T_2)&=\frac{q_2^{1/2}}{2}\begin{pmatrix}1-q_2^{-1}&1+q_2^{-1}\\
1+q_2^{-1}&1-q_2^{-1}\end{pmatrix}&\pi(x^{\omega_2})&=\mathrm{diag}(q_1^{-1/2},-q_1^{-1/2}).
\end{align*}
Note that $\pi(x^{\alpha_1^{\vee}})=q_1^{-1}I$ and $\pi(x^{\alpha_2^{\vee}})=-I$, and it follows that the restriction $\pi|_{\scH_Q}$ is not irreducible (indeed $\pi|_{\scH_Q}$ is the direct sum of the representations $\pi^4$ and $\pi^5$ from Section~\ref{sect:C2Q}). This does not contradict the irreducibility statement of Theorem~\ref{thm:repmain}, because the calibration graph changes if we use the lattice $Q$ instead of~$P$ (see Example~\ref{ex:2}).
\end{Example}

\begin{Example}\label{ex:2}
Now let $\scH$ be a $\tilde{C}_2$ Hecke algebra with $L=Q$. Let $t\in\Hom(Q,\CC^{\times})$ be the character with $t^{\alpha_1^{\vee}}=q_1^{-1}$ and $t^{\alpha_2^{\vee}}=-1$ (note the similarity to Example~\ref{ex:1}). Then $N(t)$ and $D(t)$ are as in Example~\ref{ex:1}. Let $J_1,J_2,J_3,J_4$ be as in Example~\ref{ex:1}, and then the sets $F_{J_i}(t)$ are as computed in Example~\ref{ex:1}. However now $s_2t=t$, and so the calibration graph is as shown in Figure~\ref{fig:calibex1}(b). Thus Theorem~\ref{thm:repmain} constructs $2$ irreducible $1$-dimensional modules, and $1$ irreducible $2$-dimensional module with central character~$t$. 
\end{Example}

\begin{figure}[!h]
\centering
\subfigure[Example~\ref{ex:1}: $L=P$]{
\begin{tikzpicture}[scale=0.9]
\node at (1,2) {$\bullet$};
\node at (2,1) {$\bullet$};
\node at (2,-1) {$\bullet$};
\node at (1,-2) {$\bullet$};
\node at (-1,2) {$\bullet$};
\node at (-2,1) {$\bullet$};
\node at (-2,-1) {$\bullet$};
\node at (-1,-2) {$\bullet$};
\draw (-1,2)--(1,2);
\draw (-2,-1)--(-2,1);
\draw (2,-1)--(2,1);
\draw (-1,-2)--(1,-2);
\node at (1.4,2) {$t$};
\node at (2.5,1) {$s_1t$};
\node at (2.7,-1) {$s_2s_1t$};
\node at (1.85,-2) {$s_1s_2s_1t$};
\node at (-1.5,2) {$s_2t$};
\node at (-2.7,1) {$s_1s_2t$};
\node at (-2.9,-1) {$s_2s_1s_2t$};
\node at (-2.1,-2) {$s_1s_2s_1s_2t$};
\end{tikzpicture}
}\hspace{2cm}
\subfigure[Example~\ref{ex:2}: $L=Q$]{
\begin{tikzpicture}[scale=0.9]
\node at (1,2) {$\bullet$};
\node at (2,1) {$\bullet$};
\node at (2,-1) {$\bullet$};
\node at (1,-2) {$\bullet$};
\draw (2,1)--(2,-1);
\node at (0,2) {$t=s_2t$};
\node at (0.7,1) {$s_1t=s_1s_2t$};
\node at (0.4,-1) {$s_2s_1t=s_2s_1s_2t$};
\node at (-1,-2) {$s_1s_2s_1t=s_1s_2s_2s_1t$};
\phantom{\node at (-1,2) {$\bullet$};
\node at (-2,1) {$\bullet$};
\node at (-2,-1) {$\bullet$};
\node at (-1,-2) {$\bullet$};}
\end{tikzpicture}}
\caption{Calibration graphs for Examples~\ref{ex:1} and~\ref{ex:2}}\label{fig:calibex1}
\end{figure}
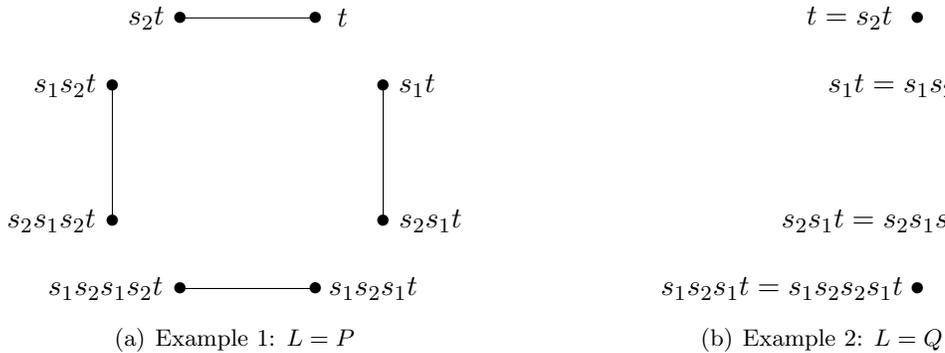

\begin{Example} Let $\scH$ be a $\tilde{G}_2$ Hecke algebra with $L=Q=P$ and with parameters $q_1$ and $q_2$ (see Section~\ref{sect:G2}). Let $t\in\Hom(Q,\CC^{\times})$ be the character with $t^{\alpha_1^{\vee}}=q_1$ and $t^{\alpha_2^{\vee}}=q_1^{-1/2}q_2^{1/2}$. If $q_1\neq q_2$ and $q_1\neq q_2^3$ then this character is regular, and we compute $N(t)^{\vee}=\{\alpha_1^{\vee},\alpha_1^{\vee}+2\alpha_2^{\vee}\}$. Thus there are $4$ choices for $J\subseteq N(t)$, and the connected components of the calibration graph are given by $\{wt\mid w\in F_J(t)\}$ for these choices of~$J$. Consider the case $J^{\vee}=\{\alpha_1^{\vee}+2\alpha_2^{\vee}\}$. We compute $F_J(t)=\{s_2s_1s_2s_1,s_1s_2s_1s_2s_1,s_2s_1s_2s_1s_2s_1\}$. The matrices for $\pi=M_J(t)$ are \begin{align*}
\pi(T_1)&=q_1^{\frac{1}{2}}\begin{pmatrix}
\frac{1-q_1^{-1}}{1- q_1^{\frac{1}{2}}q_2^{-\frac{3}{2}}}&
\frac{1- q_1^{-\frac{3}{2}}q_2^{\frac{3}{2}}}{1- q_1^{-\frac{1}{2}}q_2^{\frac{3}{2}}}&0\\
\frac{1- q_1^{-\frac{1}{2}}q_2^{-\frac{3}{2}}}{1- q_1^{\frac{1}{2}}q_2^{-\frac{3}{2}}}&
\frac{1-q_1^{-1}}{1- q_1^{-\frac{1}{2}}q_2^{\frac{3}{2}}}&0\\
0&0&-q_1^{-1}
\end{pmatrix}&
\pi(T_2)&=q_2^{\frac{1}{2}}\begin{pmatrix}
-q_2^{-1}&0&0\\
0&\frac{1-q_2^{-1}}{1- q_1^{\frac{1}{2}}q_2^{-\frac{1}{2}}}&
\frac{1- q_1^{\frac{1}{2}}q_2^{-\frac{3}{2}}}{1- q_1^{-\frac{1}{2}}q_2^{\frac{1}{2}}}\\
0&\frac{1- q_1^{-\frac{1}{2}}q_2^{-\frac{1}{2}}}{1- q_1^{\frac{1}{2}}q_2^{-\frac{1}{2}}}&\frac{1-q_2^{-1}}{1- q_1^{-\frac{1}{2}}q_2^{\frac{1}{2}}}
\end{pmatrix}\\
\pi(x^{\alpha_1^{\vee}})&=\mathrm{diag}( q_1^{-\frac{1}{2}}q_2^{\frac{3}{2}}, q_1^{\frac{1}{2}}q_2^{-\frac{3}{2}},q_1^{-1})&\pi(x^{\alpha_2^{\vee}})&=\mathrm{diag}(q_2^{-1}, q_1^{-\frac{1}{2}}q_2^{\frac{1}{2}}, q_1^{\frac{1}{2}}q_2^{-\frac{1}{2}})
\end{align*}
The construction breaks down when $q_1=q_2$ or when $q_1=q_2^3$. These cases can be dealt with by a suitable change of basis in the module $M_J(t)$. Let
\begin{align*}
A&=\begin{pmatrix}
1&0&0\\
0&1&-q_1^{\frac{1}{2}}q_2^{-\frac{1}{2}}\\
0&-1&1\end{pmatrix}&B&=\begin{pmatrix}1&-q_1^{\frac{1}{2}}q_2^{-\frac{3}{2}}&0\\
-1&1&0\\
0&0&1-q_1^{\frac{1}{2}}q_2^{-\frac{3}{2}}
\end{pmatrix}.
\end{align*}
After conjugating each representing matrix by $A$ (respectively $B$) it is observed that the resulting matrices are defined at $q_1=q_2$ (respectively $q_1=q_2^3$). Setting $q_1=q_2=q$ (respectively $q_1=q^3$ with $q_2=q$) gives a (non-calibrated) irreducible representation of the algebra $\scH(q,q)$ (respectively the algebra $\scH(q,q^3)$). For example, the matrices in the $q_1=q_2=q$ case become
\begin{align*}
\pi(T_1)&=q^{\frac{1}{2}}\begin{pmatrix}
1&\frac{3}{q-1}&\frac{3}{q-1}\\
\frac{q+1}{q}&\frac{2q+1}{q(q-1)}&\frac{3}{q-1}\\
-\frac{q+1}{q}&-\frac{3}{q-1}&-\frac{4q-1}{q(q-1)}\end{pmatrix}
&
\pi(T_2)&=q^{\frac{1}{2}}\begin{pmatrix}
-q^{-1}&0&0\\
0&1&0\\
0&-2q^{-1}&-q^{-1}\end{pmatrix}\\
\pi(x^{\alpha_1^{\vee}})&=\begin{pmatrix}
q&0&0\\
0&-2q^{-1}&-3q^{-1}\\
0&3q^{-1}&4q^{-1}\end{pmatrix}&
\pi(x^{\alpha_2^{\vee}})&=\begin{pmatrix}
q^{-1}&0&0\\
0&3&2\\
0&-2&-1\end{pmatrix}.
\end{align*}
\end{Example}

\begin{Example} Let $\scH$ be a $\tilde{C}_2$ affine Hecke algebra with either $L=Q$ or $L=P$ and with parameters $q_1$ and~$q_2$ (see Sections~\ref{sect:C2Q} and \ref{sect:C2P}). Let $t\in\Hom(L,\CC^{\times})$ be a character with $t^{\alpha_1^{\vee}}=q_1^{-1}$ and $t^{\alpha_2^{\vee}}=q_2$. If $q_1\neq q_2$ and $q_1\neq q_2^2$ then the character $t$ is regular, since $t^{\alpha_1^{\vee}+\alpha_2^{\vee}}=q_1^{-1}q_2$ and $t^{\alpha_1^{\vee}+2\alpha_2^{\vee}}=q_1^{-1}q_2^2$. 
Thus we compute $N(t)=\{\alpha_1,\alpha_2\}$ and $D(t)=\emptyset$. There are $4$ choices for $J\subseteq N(t)$, namely $J_1=\emptyset$, $J_2=\{\alpha_1\}$, $J_3=\{\alpha_2\}$, and $J_4=\{\alpha_1,\alpha_2\}$. We compute
\begin{align*}
F_{J_1}(t)&=\{1\},&F_{J_2}(t)&=\{s_1,s_2s_1,s_1s_2s_1\},&F_{J_3}(t)&=\{s_2,s_1s_2,s_2s_1s_2\},&F_{J_4}(t)&=\{s_1s_2s_1s_2\}.
\end{align*}
The sets $\{wt\mid w\in F_{J_i}(t)\}$ with $i=1,2,3,4$ are the connected components of the calibration graph of~$t$. Thus there are $4$ irreducible modules $M_{J_i}(t)$ ($i=1,2,3,4$) with central character~$t$, with dimensions $1,3,3,1$ respectively.

Consider the module $M_{J_3}(t)$ (this module will appear in the Plancherel Theorem for $\tilde{C}_2$). The matrices of $T_1,T_2,x^{\alpha_1^{\vee}}$ and $x^{\alpha_2^{\vee}}$ relative to the basis $e_{s_2t},e_{s_1s_2t},e_{s_2s_1s_2t}$ are
\begin{align*}
\pi(T_1)&=q_1^{\frac{1}{2}}\begin{pmatrix}
\frac{1-q_1^{-1}}{1-q_1q_2^{-2}}&
\frac{1-q_1^{-2}q_2^2}{1-q_1^{-1}q_2^2}&
0\\
\frac{1-q_2^{-2}}{1-q_1q_2^{-2}}&
\frac{1-q_1^{-1}}{1-q_1^{-1}q_2^2}&
0\\
0&0&-q_1^{-1}
\end{pmatrix}&
\pi(T_2)&=q_2^{\frac{1}{2}}
\begin{pmatrix}
-q_2^{-1}&0&0\\
0&
\frac{1-q_2^{-1}}{1-q_1q_2^{-1}}&
\frac{1-q_1^{-1}}{1-q_1^{-1}q_2}\\
0&
\frac{1-q_1q_2^{-2}}{1-q_1q_2^{-1}}&
\frac{1-q_2^{-1}}{1-q_1^{-1}q_2}
\end{pmatrix}\\
\pi(x^{\alpha_1^{\vee}})&=\mathrm{diag}(q_1^{-1}q_2^2,q_1q_2^{-2},q_1^{-1})&\pi(x^{\alpha_2^{\vee}})&=\mathrm{diag}(q_2^{-1},q_1^{-1}q_2,q_1q_2^{-1}).
\end{align*}
If $L=P$ then $\omega_1=\alpha_1^{\vee}+\alpha_2^{\vee}$ and $\omega_2=\alpha_1^{\vee}/2+\alpha_2^{\vee}$. Thus there are $2$ characters $t\in\Hom(P,\CC^{\times})$ with $t^{\alpha_1^{\vee}}=q_1^{-1}$ and $t^{\alpha_2^{\vee}}=q_2$, specifically $t^{\omega_1}=q_1^{-1}q_2$ and $t^{\omega_2}=\pm q_1^{-1/2}q_2$. The corresponding matrices for $x^{\omega_1}$ and $x^{\omega_2}$ are
\begin{align*}
\pi(x^{\omega_1})&=\mathrm{diag}(q_1^{-1}q_2,q_2^{-1},q_2^{-1})&\pi(x^{\omega_2})&=\pm\mathrm{diag}(q_1^{-1/2},q_1^{-1/2},q_1^{1/2}q_2^{-1}).
\end{align*}
In the cases $q_1=q_2$ or $q_1=q_2^2$ an analogous computation to that in Example 2 can be used to construct (non-calibrated) irreducible representations of $\scH(q,q)$ and $\scH(q,q^2)$.
\end{Example}

\subsection{Characters}\label{sect:characters}

We conclude this section with some observations about characters that will be used for the Plancherel Theorems. Let $f_t(h)$ be as in~(\ref{eq:f}).

\begin{lemma}\label{lem:milage} Let $\pi$ be an irreducible representation of $\scH_L$ with central character $t$, and suppose that the character $\chi$ of $\pi$ satisfies 
$$
\chi(\tau_wx^{\lambda})=\delta_{w,1}\sum_{v\in W_0}k_v(vt)^{\lambda}\qquad\textrm{for all $w\in W_0$ and $\lambda\in L$},
$$
for some numbers $k_v\in\CC$. Then, with $f_t$ as in~(\ref{eq:f}), if $t$ is regular we have
$$
\chi(h)=\sum_{v\in W_0}k_vf_{vt}(h)\qquad\textrm{for all $h\in\scH_L$}.
$$
\end{lemma}

\begin{proof} Since $t$ is regular each $f_{vt}(h)$ with $v\in W_0$ and $h\in\scH_L$ is defined. 
From Remark~\ref{rem:explicit} and the hypothesis we have $\chi(h)=\sum_{v\in W_0}k_vf_{vt}(h)$ for all $h\in\scH_L'$, where~$\scH_L'$ is the subalgebra of $\scH_L$ with basis $\{\tau_w x^{\lambda}\mid w\in W_0,\lambda\in L\}$. Let $\Delta(x)=\prod_{\alpha\in R_0}(1-x^{-\alpha^{\vee}})=d(x)d(x^{-1})$. An induction using the formula $(1-x^{-\alpha_i^{\vee}})T_i=\tau_i+a_i(x)$ shows that
$
\Delta(x)^{\ell(w)}T_w\in\scH_L'$ for all $w\in W_0$. Thus $
\Delta(x)^{\ell(w)}T_wx^{\lambda}\in\scH_L'$ for all $w\in W_0$ and $\lambda\in L$.
Since $\Delta(x)\in\CC[L]^{W_0}$ is central and $\chi$ is irreducible we have 
$$
\Delta(t)^{\ell(w)}\chi(T_w x^{\lambda})=\chi(\Delta(x)^{\ell(w)}T_wx^{\lambda})=\sum_{v\in W_0}k_vf_{vt}(\Delta(x)^{\ell(w)}T_wx^{\lambda})=\Delta(t)^{\ell(w)}\sum_{v\in W_0}k_vf_{vt}(T_wx^{\lambda}).
$$
We can divide through by $\Delta(t)^{\ell(w)}$ since $t$ is regular.
\end{proof}

\begin{prop}\label{prop:prelim_main}
Let $\chi_t$ be the character of the principal series representation $M(t)$ of $\scH_L$ with central character~$t$. Then
\begin{align}\label{eq:rg}
\chi_t(h)=\sum_{w\in W_0}f_{wt}(h)\qquad\textrm{for all $h\in\scH_L$}
\end{align}
where the right hand side has an analytic continuation (for fixed $h\in\scH_L$) to all $t\in\Hom(L,\CC^{\times})$.
\end{prop}

\begin{proof}
Suppose first that $D(t)=\emptyset$ and that $M(t)$ is irreducible (see Theorem~\ref{thm:kato}). Since $D(t)=\emptyset$ the module $M(t)$ has basis $\{\tau_w\otimes v_t\mid w\in W_0\}$. To see this note that if $w=s_{i_1}\cdots s_{i_{\ell}}$ is reduced then the Bernstein relation gives
\begin{align*}
\tau_w\otimes v_t&=\bigg[\prod_{\alpha\in R(w^{-1})}(1-t^{\alpha^{\vee}})\bigg](T_w\otimes v_t)+\textrm{lower terms},
\end{align*}
where `lower terms' is a linear combination of terms $T_v\otimes v_t$ with $v<w$ in Bruhat order. Thus if $D(t)=\emptyset$ then each basis element $T_w\otimes v_t$ of $M(t)$ can be written in terms of the elements $\{\tau_{w}\otimes v_t\mid w\in W_0\}$.

From (\ref{eq:tausquared}) we see that the diagonal entries of the matrix for $\tau_w$ are all~$0$. The matrix for $x^{\lambda}$ is diagonal with entries $(wt)^{\lambda}$ ($w\in W_0$) on the diagonal. Therefore
$$
\chi_t(\tau_w x^{\lambda})=\delta_{w,1}\sum_{v\in W_0}(vt)^{\lambda}\qquad\textrm{for all $w\in W_0$ and $\lambda\in L$}.
$$  
Hence Lemma~\ref{lem:milage} gives (\ref{eq:rg}). 

The cases where $D(t)\neq \emptyset$ or $M(t)$ is not irreducible are obtained as follows. For fixed $h\in\scH$ the the character $\chi_t(h)$ is, by construction, a linear combination of $\{t^{\lambda}\mid \lambda\in L\}$ and is defined for all $t\in\Hom(L,\CC^{\times})$. The right hand side of (\ref{eq:rg}) is a rational function in~$t$. Thus the singularities of this rational function are removable singularities (even though each individual summand may have singularities).
\end{proof}

\begin{prop}\label{prop:calc}
Suppose that $t$ is a regular character. Let $J\subseteq N(t)$, and let $M_J(t)$ be the module constructed in Theorem~\ref{thm:repmain}. Then
$$
\chi(h)=\sum_{w\in F_J(t)}f_{wt}(h)\qquad\textrm{for all $h\in\scH$}.
$$
\end{prop}

\begin{proof} Since $\tau_i\cdot e_{wt}=q_i^{1/2}n_i(t)\,e_{s_iwt}$ we see that the diagonal entries of the matrix for $\tau_w$ are~$0$. Since the matrix representing $x^{\lambda}$ is diagonal it follows that
$$
\chi(\tau_v x^{\lambda})=\delta_{v,1}\sum_{w\in F_J(t)}(wt)^{\lambda}\qquad\textrm{for all $v\in W_0$ and $\lambda\in L$},
$$
and the result follows from Lemma~\ref{lem:milage}.
\end{proof}

\begin{lemma}\label{lem:1dim} Let $\pi$ be a $1$-dimensional representation of $\scH_L$ with regular central character~$t$. Then
$$
\chi(h)=f_t(h)\qquad\textrm{for all $h\in\scH_L$},
$$
unless $\scH_L$ is of type $\tilde{C}_n$ with $\pi(x^{\alpha_n^{\vee}})=-1$. In this case there is a $1$-dimensional representation $\pi'$ defined by $\pi'(x^{\lambda})=\pi(x^{\lambda})$ for all $\lambda\in L$, $\pi'(T_i)=\pi(T_i)$ for all $i\neq n$, and $\pi'(T_n)=q_n^{1/2}$ (respectively $-q_n^{-1/2}$) if $\pi(T_n)=-q_n^{-1/2}$ (respectively $q_n^{1/2}$). Then
$$
\frac{\chi(h)+\chi'(h)}{2}=f_t(h)\qquad\textrm{for all $h\in\scH_L$}.
$$
\end{lemma}

\begin{proof}
By direct analysis of the defining relations (\ref{eq:rel1})-(\ref{eq:rel4}) one sees that the central character~$t$ of a $1$-dimensional representation necessarily has $n_i(t)n_i(t^{-1})=0$, except in the $\tilde{C}_n$ case with $\pi(x^{\alpha_n^{\vee}})=-1$. Excluding this case for the moment, it follows from (\ref{eq:tausquared}) that $\pi(\tau_i)=0$ and hence $\pi(\tau_wx^{\lambda})=\delta_{w,1}t^{\lambda}$ for all $w\in W_0$ and $\lambda\in L$. Since $t$ is assumed to be regular, Lemma~\ref{lem:milage} gives $\chi(h)=f_t(h)$ for all $h\in\scH_L$.

Now consider the $\tilde{C}_2$ case with $\pi(x^{\alpha_n^{\vee}})=-1$. Let $\pi'$ be the companion representation defined in the statement of the lemma. The proof of Lemma~\ref{lem:milage} applied to the representation $\pi\oplus\pi'$ proves the result. The fact that $\pi\oplus\pi'$ is not irreducible does not effect the proof of Lemma~\ref{lem:milage} because the centre $\CC[L]^{W_0}$ of $\scH_L$ acts by the same scalar on each of $\pi$ and $\pi'$.
\end{proof}

\begin{remark}\label{rem:regular} In Proposition~\ref{prop:calc} and Lemma~\ref{lem:1dim} we assumed that $t$ is a regular central character. In general these results are false for non-regular central characters, even if each term $f_t(h)$ is defined. For example consider the $\tilde{G}_2$ case with $t\in\Hom(Q,\CC^{\times})$ given by $t^{\alpha_1^{\vee}}=q_1$, $t^{\alpha_2^{\vee}}=q_2^{-1}$. If $q_1\neq q_2$ and $q_1\neq q_2^2$ and $q_1^2\neq q_2^3$ and $q_1\neq q_2^3$ then this central character is regular, and by Lemma~\ref{lem:1dim} we have $f_t(h)=\chi^4(h)$ for all $h\in\scH$, where $\chi^4$ is the $1$-dimensional representation of $\scH$ listed in Section~\ref{sect:G2}. Suppose that $q_1=q_2=q$. A calculation similar to Remark~\ref{rem:explicit} shows that $f_t(h)$ is defined for all $h\in\scH$, and that $f_t(T_1T_2T_1)=q^{1/2}$. But $\chi^4(T_1T_2T_1)=-q^{1/2}$.
\end{remark}

\section{The Plancherel Theorem}\label{sect:plancherel}

In this section we state and prove the Plancherel Theorem for each irreducible affine Hecke algebra of rank~$1$ or rank~$2$. In each case we give the generators and relations for the algebra, and construct the representations that appear in the Plancherel Theorem (see the appendix for some explicit matrices). We then state the Plancherel Theorem, and give a proof starting from the trace generating function formula~(\ref{eq:prelim_main}). The proof consists of performing a series of contour shifts and Proposition~\ref{prop:prelim_main} to write~(\ref{eq:prelim_main}) as
\begin{align}\label{eq:integrandtest}
\Tr(h)=\frac{1}{|W_0|q_{w_0}}\int_{\TT^n}\frac{\chi_t(h)}{|c(t)|^2}\,dt+\textrm{lower terms}
\end{align}
where the lower order terms are integrals over lower dimensional tori. Then the lower terms are matched up with lower dimensional representations of the Hecke algebra using Proposition~\ref{prop:calc} and Lemma~\ref{lem:1dim}. 

Throughout this section we assume that $q_0,q_1,\ldots,q_n>1$. The possible pairs $(R,L)$ with $R$ an irreducible rank~$2$ root system and $L$ a $\ZZ$-lattice with $Q\subseteq L\subseteq P$ are $(R,L)=(A_2,Q)$, $(A_2,P)$, $(C_2,Q)$, $(C_2,P)$, $(G_2,Q)$, and $(BC_2,Q)$.

\subsection{The rank~1 algebras}

\noindent(1) The $\tilde{A}_1(q)$, $L=Q$, algebra has generators $T=T_1$ and $x=x^{\alpha_1^{\vee}}$ with relations
\begin{align*}
T^2=1+(q^{\frac{1}{2}}-q^{-\frac{1}{2}})T,\qquad
Tx=x^{-1}T+(q^{\frac{1}{2}}-q^{-\frac{1}{2}})(1+x).
\end{align*}
Let $\pi_t=\mathrm{Ind}_{\CC[Q]}^{\scH}(\CC v_t)$ be the principal series representation with central character~$t\in\CC^{\times}$, where $x\cdot v_t=t v_t$. Let $\pi$ be the $1$-dimensional representation of $\scH$ with 
$$
\pi(T)=-q^{-\frac{1}{2}}\qquad\textrm{and}\qquad \pi(x)=q^{-1}.
$$
Let $\chi_t$ be the character of $\pi_t$ and let $\chi$ be the character of $\pi$.
\smallskip

\noindent(2) The $\tilde{A}_1(q)$, $L=P$, algebra has generators $T=T_1$ and $x=x^{\omega_1}$ with relations
\begin{align*}
T^2=1+(q^{\frac{1}{2}}-q^{-\frac{1}{2}})T,\qquad
Tx=x^{-1}T+(q^{\frac{1}{2}}-q^{-\frac{1}{2}})x.
\end{align*}
Let
$
\pi_t=\mathrm{Ind}_{\CC[P]}^{\scH}(\CC v_t)
$
be the principal series representation with central character~$t\in\CC^{\times}$, where $x\cdot v_t=t v_t$. Let $\pi^1$ and $\pi^2$ be the $1$-dimensional representations of $\scH$ with 
$$
\pi^1(T)=-q^{-\frac{1}{2}},\quad\pi^1(x)=q^{-\frac{1}{2}},\qquad\textrm{and}\qquad\pi^2(T)=-q^{-\frac{1}{2}},\quad \pi^2(x)=-q^{-\frac{1}{2}}.
$$
Let $\chi_t$, $\chi^1$, and $\chi^2$ be the characters of $\pi_t$, $\pi^1$, and $\pi^2$ (respectively).
\smallskip

\noindent(3) The $\tilde{BC}_1(q_0,q_1)$, $L=Q$, algebra has generators $T=T_1$, $x=x^{\alpha_1^{\vee}/2}$ with relations
\begin{align*}
T^2=1+(q_1^{\frac{1}{2}}-q_1^{-\frac{1}{2}})T,\qquad
Tx=x^{-1}T+(q_1^{\frac{1}{2}}-q_1^{-\frac{1}{2}})x+(q_0^{\frac{1}{2}}-q_0^{-\frac{1}{2}}).
\end{align*}
Let $\pi_t=\mathrm{Ind}_{\CC[Q]}^{\scH}(\CC v_t)$ be the principal series representation with central character~$t\in\CC^{\times}$, where $x\cdot v_t=t v_t$. Let $\pi^{1}$, $\pi^{2}$ and $\pi^{3}$ be the $1$-dimensional representations of $\scH$ with 
$$
\begin{aligned}
\pi^1(T)&=-q_1^{-\frac{1}{2}}\\
\pi^1(x)&=q_0^{-\frac{1}{2}}q_1^{-\frac{1}{2}}
\end{aligned}\qquad
\begin{aligned}
\pi^2(T)&=-q_1^{-\frac{1}{2}}\\
\pi^2(x)&=-q_0^{\frac{1}{2}}q_1^{-\frac{1}{2}}
\end{aligned}\qquad
\begin{aligned}
\pi^{3}(T)&=q_1^{\frac{1}{2}}\\
\pi^3(x)&=-q_0^{-\frac{1}{2}}q_1^{\frac{1}{2}}
\end{aligned}
$$
Let $\chi_t$, $\chi^1$, $\chi^2$, and $\chi^3$ be the characters of $\pi_t$, $\pi^1$, $\pi^2$, and $\pi^3$ (respectively).

\begin{thm} Let $h\in\scH$. In the cases (1), (2) and (3) above we have, respectively:
\begin{align*}
\Tr(h)&=\frac{1}{2q}\int_{\TT}\frac{\chi_t(h)}{|c(t)|^2}\,dt+\frac{q-1}{q+1}\chi(h)\\
\Tr(h)&=\frac{1}{2q}\int_{\TT}\frac{\chi_t(h)}{|c(t)|^2}\,dt+\frac{q-1}{2(q+1)}\left(\chi^{1}(h)+\chi^{2}(h)\right)\\
\Tr(h)&=\frac{1}{2q_1}\int_{\TT}\frac{\chi_t(h)}{|c(t)|^2}\,dt+\frac{q_0q_1-1}{(q_0+1)(q_1+1)}\chi^{1}(h)+\frac{|q_0-q_1|}{(q_0+1)(q_1+1)}\times\begin{cases}\chi^2(h)&\textrm{if $q_0<q_1$}\\
\chi^3(h)&\textrm{if $q_1<q_0$},
\end{cases}
\end{align*}
where the $c$-functions are (respectively)
\begin{align*}
c(t)&=\frac{1-q^{-1}t^{-1}}{1-t^{-1}},&
c(t)&=\frac{1-q^{-1}t^{-2}}{1-t^{-2}},&
c(t)&=\frac{(1-q_0^{-\frac{1}{2}}q_1^{-\frac{1}{2}}t^{-1})(1+q_0^{\frac{1}{2}}q_1^{-\frac{1}{2}}t^{-1})}{1-t^{-2}}.
\end{align*}
\end{thm}
 
\begin{proof} Let us prove the $\tilde{BC}_1(q_0,q_1)$ case. If $q_0=q_1$ there is some simplification, so suppose that $q_0\neq q_1$. Write $g(t)=g_t(h)$ and $f(t)=f_t(h)$. From (\ref{eq:prelim_main}) we have
$$
\Tr(h)=\frac{1}{q_1}\int_{q_0^{-\frac{1}{2}}q_1^{-\frac{1}{2}}a\TT}\frac{f(t)}{c(t)c(t^{-1})}\,dt
$$
where $0<a<1$. Note that the integrand has at most removable singularities on~$t\in\TT$, and that the poles of the integrand that lie between the contours $q_0^{-\frac{1}{2}}q_1^{-\frac{1}{2}}a\TT$ and $\TT$ are at $t=q_0^{-\frac{1}{2}}q_1^{-\frac{1}{2}}$, $t=-q_0^{\frac{1}{2}}q_1^{-\frac{1}{2}}$ (in the case that $q_0<q_1$) and $t=-q_0^{-\frac{1}{2}}q_1^{\frac{1}{2}}$ (in the case that $q_1<q_0$). Computing residues (using $dt=\frac{1}{2\pi}d\theta=\frac{1}{2\pi i}\frac{dz}{z}$) gives
$$
\Tr(h)=
\frac{1}{q_1}\int_{\TT}\frac{f(t)}{|c(t)|^2}\,dt+\frac{q_0q_1-1}{(q_0+1)(q_1+1)}f(q_0^{-\frac{1}{2}}q_1^{-\frac{1}{2}})+\frac{|q_0-q_1|}{(q_0+1)(q_1+1)}\cdot
\begin{cases}
f(-q_0^{\frac{1}{2}}q_1^{-\frac{1}{2}})&\textrm{if $q_0<q_1$}\\
f(-q_0^{-\frac{1}{2}}q_1^{\frac{1}{2}})&\textrm{if $q_1<q_0$}.
\end{cases}
$$
Using Proposition~\ref{prop:prelim_main} we have
$$
\frac{1}{q_1}\int_{\TT}\frac{f(t)}{|c(t)|^2}\,dt=\frac{1}{2q_1}\int_{\TT}\frac{f(t)+f(t^{-1})}{|c(t)|^2}\,dt=\frac{1}{2q_1}\int_{\TT}\frac{\chi_t(h)}{|c(t)|^2}\,dt,
$$
and Lemma~\ref{lem:1dim} gives $f(q_0^{-\frac{1}{2}}q_1^{-\frac{1}{2}})=\chi^1(h)$, $f(-q_0^{\frac{1}{2}}q_1^{-\frac{1}{2}})=\chi^2(h)$, and $f(-q_0^{-\frac{1}{2}}q_1^{\frac{1}{2}})=\chi^3(h)$.
\end{proof}

\subsection{The $\tilde{A}_2(q)$ algebras with $L=Q$}\label{sect:A2Q}

This case is treated in \cite{PS}, and so we will just state the result here. The coroot system is 
$
R=\pm\{\alpha_1^{\vee},\alpha_2^{\vee},\alpha_1^{\vee}+\alpha_2^{\vee}\}
$.
The affine Hecke algebra has generators $T_1,T_2,x_1=x^{\alpha_1^{\vee}},x_2=x^{\alpha_2^{\vee}}$ and relations
\begin{align*}
T_1^2&=1+(q^{\frac{1}{2}}-q^{-\frac{1}{2}})T_1& T_1x_1&=x_1^{-1}T_1+(q^{\frac{1}{2}}-q^{-\frac{1}{2}})(1+x_1) & T_1T_2T_1&=T_2T_1T_2\\
T_2^2&=1+(q^{\frac{1}{2}}-q^{-\frac{1}{2}})T_2 & T_2x_2&=x_2^{-1}T_2+(q^{\frac{1}{2}}-q^{-\frac{1}{2}})(1+x_2)& x_1x_2&=x_2x_1\\
T_1x_2&=x_1x_2T_1^{-1} & T_2x_1&=x_1x_2T_2^{-1}.
\end{align*}

Let $\pi_t=\mathrm{Ind}_{\CC[Q]}^{\scH}(\CC v_t)$ be the principal series representation of the affine Hecke algebra~$\scH$ with central character~$t=(t_1,t_2)\in(\CC^{\times})^2$, where $\CC v_t$ is the $1$-dimensional representation of $\CC[Q]$ with $x_1\cdot v_t=t_1v_t$ and $x_2\cdot v_t=t_2v_t$. 

Let $\scH_1$ be the subalgebra of $\scH$ generated by $T_1,x_1$ and~$x_2$. Let $s\in\CC^{\times}$ and let $\CC u_s$ be the $1$-dimensional representation of $\scH_1$ with
\begin{align*}
T_1\cdot u_s&=-q^{-\frac{1}{2}}u_s,& x_1\cdot u_s&=q^{-1}u_s,& x_2\cdot u_s&=q^{\frac{1}{2}}s u_s.
\end{align*}
Let
$
\pi_s^{1}=\mathrm{Ind}_{\scH_1}^{\scH}(\CC u_s)
$ be the induced representation of $\scH$. 

Let $\pi^{2}$ be the $1$-dimensional representation of $\scH$ with 
\begin{align*}
\pi^{2}(T_1)&=-q^{-\frac{1}{2}},&\pi^{2}(T_2)&=-q^{-\frac{1}{2}},&\pi^{2}(x_1)&=q^{-1},&\pi^{2}(x_2)&=q^{-1}.
\end{align*}
Let $\chi_t$, $\chi_s^1$, and $\chi^2$ be the characters of $\pi_t$, $\pi_s^1$, and $\pi^2$ (respectively).

\begin{thm} for all $h\in\scH$ we have 
$$
\Tr(h)=\frac{1}{6q^3}\int_{\TT^2}\frac{\chi_t(h)}{|c(t)|^2}\,dt+\frac{(q-1)^2}{q^2(q^2-1)}\int_{\TT}\frac{\chi_s^{1}(h)}{|c_1(s)|^2}\,ds+\frac{(q-1)^3}{q^3-1}\chi^{2}(h),
$$
where 
\begin{align*}
c(t)&=\frac{(1-q^{-1}t_1^{-1})(1-q^{-1}t_2^{-1})(1-q^{-1}t_1^{-1}t_2^{-1})}{(1-t_1^{-1})(1-t_2^{-1})(1-t_1^{-1}t_2^{-1})},&c_1(s)&=\frac{1-q^{-\frac{3}{2}}s^{-1}}{1-q^{\frac{1}{2}}s^{-1}}.
\end{align*}
\end{thm}

\subsection{The $\tilde{A}_2(q)$ algebras with $L=P$}

The root system is as in Section~\ref{sect:A2Q}. The fundamental coweights are $\omega_1=\frac{2}{3}\alpha_1^{\vee}+\frac{1}{3}\alpha_2^{\vee}$ and $\omega_2=\frac{1}{3}\alpha_1^{\vee}+\frac{2}{3}\alpha_2^{\vee}$, and the coweight lattice is $P=\ZZ\omega_1+\ZZ\omega_2$. The affine Hecke algebra is generated by $T_1$, $T_2$, $x_1=x^{\omega_1}$ and $x_2=x^{\omega_2}$ with relations
\begin{align*}
T_1^2&=1+(q^{\frac{1}{2}}-q^{-\frac{1}{2}})T_1& T_1x_1&=x_1^{-1}x_2T_1+(q^{\frac{1}{2}}-q^{-\frac{1}{2}})x_1 & T_1T_2T_1&=T_2T_1T_2\\
T_2^2&=1+(q^{\frac{1}{2}}-q^{-\frac{1}{2}})T_2 & T_2x_2&=x_1x_2^{-1}T_2+(q^{\frac{1}{2}}-q^{-\frac{1}{2}})x_2& x_1x_2&=x_2x_1\\
T_1x_2&=x_2T_1 & T_2x_1&=x_1T_2.
\end{align*}

Let $\pi_t=\mathrm{Ind}_{\CC[P]}^{\scH}(\CC v_t)$ be the principal series representation of the affine Hecke algebra $\scH$ with central character~$t=(t_1,t_2)\in(\CC^{\times})^2$, where $\CC v_t$ is the $1$-dimensional representation of $\CC[P]$ with $x_1\cdot v_t=t_1v_t$ and $x_2\cdot v_t=t_2v_t$. 

Let $\scH_1$ be the subalgebra generated by $T_1,x_1$ and $x_2$. Let $s\in\CC^{\times}$, and let $\pi_s^{1}=\scH\otimes_{\scH_1}(\CC u_s)$ be the $3$-dimensional representation of $\scH$ induced from the $1$-dimensional representation $\CC u_s$ of $\scH_1$ given by
\begin{align*}
T_1\cdot u_s&=-q^{-\frac{1}{2}}u_s,& x_1\cdot u_s&=q^{-\frac{1}{2}}su_s,& x_2\cdot u_s&=s^2u_s.
\end{align*}
The module $\pi_s^1$ has basis $\{1\otimes u_s,T_2\otimes u_s,T_1T_2\otimes u_s\}$ and support $\mathrm{supp}\,\pi_s^1=\{t,s_2t,s_1s_2t\}$, where $t\in\Hom(P,\CC^{\times})$ is the character with $(t^{\omega_1},t^{\omega_2})=(q^{-1/2}s,s^2)$. It is not hard to show (see the proof of Lemma~\ref{lem:C2Q}) that the character of $\pi_s^1$ satisfies
\begin{align}\label{eq:A2_char}
\chi_s(h)=f_t(h)+f_{s_2t}(h)+f_{s_1s_2t}(h)\qquad\textrm{for all $s\in\CC^{\times}$ and all $h\in\scH$}.
\end{align}

Let $\pi^{2}$, $\pi^{3}$ and $\pi^{4}$ be the $1$-dimensional representations of $\scH$ given by (where $\omega=e^{2\pi i/3}$)
\begin{align*}
\pi^2(T_1)&=-q^{-\frac{1}{2}} & \pi^2(T_2)&=-q^{-\frac{1}{2}} & \pi^2(x_1)&=q^{-1} & \pi^2(x_2)&=q^{-1}\\
\pi^3(T_1)&=-q^{-\frac{1}{2}} & \pi^3(T_2)&=-q^{-\frac{1}{2}} & \pi^3(x_1)&=\omega q^{-1} & \pi^3(x_2)&=\omega^{-1}q^{-1}\\
\pi^4(T_1)&=-q^{-\frac{1}{2}} & \pi^4(T_2)&=-q^{-\frac{1}{2}} & \pi^4(x_1)&=\omega^{-1}q^{-1} & \pi^4(x_2)&=\omega q^{-1}.
\end{align*}
Let $\chi_t$, $\chi_s^1$, $\chi^2$, $\chi^3$ and $\chi^4$ be the characters of $\pi_t$, $\pi^1_s$, $\pi^2$, $\pi^3$, and $\pi^4$ (respectively).

\begin{thm} For all $h\in\scH$ we have
$$
\tau(h)=\frac{1}{6q^3}\int_{\TT^2}\frac{\chi_t(h)}{|c(t)|^2}\,dt+\frac{(q-1)^2}{q^2(q^2-1)}\int_{\TT}\frac{\chi_s^{1}(h)}{|c_1(s)|^2}\,ds+\frac{(q-1)^3}{3(q^3-1)}\left(\chi^{2}(h)+\chi^{3}(h)+\chi^{4}(h)\right)
$$
where 
\begin{align*}
c(t)&=\frac{(1-q^{-1}t_1^{-2}t_2)(1-q^{-1}t_1t_2^{-2})(1-q^{-1}t_1^{-1}t_2^{-1})}{(1-t_1^{-2}t_2)(1-t_1t_2^{-2})(1-t_1^{-1}t_2^{-1})},& c_1(s)&=\frac{1-q^{-\frac{3}{2}}s^{-3}}{1-q^{\frac{1}{2}}s^{-3}}.
\end{align*}
\end{thm}

\begin{proof}
The series $G_t(h)$ converges whenever $|t^{\alpha_1^{\vee}}|,|t^{\alpha_2^{\vee}}|<q^{-1}$, and hence the series converges whenever $|t_1|,|t_2|< q^{-1}$, where $t_1=t^{\omega_1}$ and $t_2=t^{\omega_2}$. Fix $h\in\scH$, and write $f_t(h)=f(t)$. Therefore
$$
\Tr(h)=\frac{1}{q^3}\int_{q^{-1}a\TT}\int_{q^{-1}b\TT}\frac{f(t)}{c(t)c(t^{-1})}\,dt_1 dt_2
$$
where $0<a,b<1$. Fix a number $0<c<1$ very close to $1$. Consider the inner integral. The $t_1$-poles of the integrand lying between the contours $q^{-1}a\TT$ and $c\TT$ are at the points where $t_1^2=q^{-1}t_2$. We compute the residues (using $dt_1=\frac{1}{2\pi i}\frac{dz_1}{z_1}$) to be
$$
\Res_{t_1=\pm q^{-1/2}t_2^{1/2}}\frac{f(t)}{c(t)c(t^{-1})}=-\frac{q(q-1)^2}{2(q^2-1)}\frac{f(\pm q^{-\frac{1}{2}}t_2^{1/2},t_2)}{c_1(\mp t_2^{1/2})c_1(\mp t_2^{-1/2})}.
$$
Using $\frac{1}{2}\int_{r\TT}(f(t^{1/2})+f(-t^{1/2}))dt=\int_{r^{1/2}\TT}f(t)dt$ it follows that
$$
\Tr(h)=\frac{1}{q^3}\int_{q^{-1}a\TT}\int_{c\TT}\frac{f(t)}{c(t)c(t^{-1})}\,dt_1dt_2+\frac{(q-1)^2}{q^2(q^2-1)}\int_{q^{-\frac{1}{2}}a^{\frac{1}{2}}\TT}\frac{f(q^{-\frac{1}{2}}s,s^2)}{c_1(s)c_1(s^{-1})}\,ds.
$$
Interchange the order of integration in the double integral. The $t_2$-poles of the integrand between the contours $q^{-1}a\TT$ to $\TT$ are at the points where $t_2^2=q_1^{-1}t_1$ and where $t_2=q^{-1}t_1^{-1}$. Computing residues gives
$$
\Tr(h)=\frac{1}{q^3}\int_{c\TT}\int_{\TT}\frac{f(t)}{|c(t)|^2}\,dt_2dt_1+\frac{(q-1)^2}{q^2(q^2-1)}(I_1+I_2+I_3),
$$
where
\begin{align*}
I_1&=\int_{q^{-\frac{1}{2}}a^{\frac{1}{2}}\TT}\frac{f(q^{-\frac{1}{2}}s,s^2)}{c_1(s)c_1(s^{-1})}\,ds& I_2&=\int_{c^{\frac{1}{2}}\TT}\frac{f(s^2,q^{-\frac{1}{2}}s)}{c_1(s)c_1(s^{-1})}\,ds& I_3&=\int_{q^{\frac{1}{2}}c\TT}\frac{f(q^{-\frac{1}{2}}s,q^{-\frac{1}{2}}s^{-1})}{c_1(s)c_1(s^{-1})}\,ds
\end{align*}
where we have set $s=t_1^{1/2}$ in $I_2$ and $s=q^{\frac{1}{2}}t_1$ in $I_3$. The $t_1$-contour in the double integral can be shifted to $\TT$ without encountering any poles. 

The plan is to shift each of the contours in $I_1,I_2$ and $I_3$ to the unit contour $\TT$. However we need to be careful with the possible singularities of $f(t)$. Therefore we write $f(t)=g(t)/d(t)$, with $g(t)$ analytic. Then the integrands of the integrals $I_1,I_2$ and $I_3$ are
\begin{align*}
\frac{f(q^{-\frac{1}{2}}s,s^2)}{c_1(s)c_1(s^{-1})}&=\frac{(1-q^{\frac{1}{2}}s^3)g(q^{-\frac{1}{2}}s,s^2)}{(1-s^{-2})(1-q^{\frac{1}{2}}s^{-1})(1-q^{-\frac{3}{2}}s^{-3})(1-q^{-\frac{3}{2}}s^3)}\\
\frac{f(s^2,q^{-\frac{1}{2}}s)}{c_1(s)c_1(s^{-1})}&=\frac{(1-q^{\frac{1}{2}}s^3)g(s^2,q^{-\frac{1}{2}}s)}{(1-s^{-2})(1-q^{\frac{1}{2}}s^{-1})(1-q^{-\frac{3}{2}}s^{-3})(1-q^{-\frac{3}{2}}s^3)}\\
\frac{f(q^{-\frac{1}{2}}s,q^{-\frac{1}{2}}s^{-1})}{c_1(s)c_1(s^{-1})}&=\frac{(1-q^{\frac{1}{2}}s^{-3})(1-q^{\frac{1}{2}}s^3)g(q^{-\frac{1}{2}}s,q^{-\frac{1}{2}}s^{-1})}{(1-q)(1-q^{-\frac{3}{2}}s^{-3})(1-q^{-\frac{3}{2}}s^3)(1-q^{\frac{1}{2}}s^{-1})(1-q^{\frac{1}{2}}s)}.
\end{align*}
In particular, the integrands of $I_1$ and $I_2$ have singularities on~$\TT$. So instead we shift all contours to $c\TT$. For the integrals $I_2$ and $I_3$ we encounter no poles, and so the shift is for free. For the integral $I_1$ we pick up simple residues at the points $s^3=q^{-\frac{3}{2}}$, and computing residues gives
$$
I_1=\int_{c\TT}\frac{f(q^{-\frac{1}{2}}s,s^2)}{c_1(s)c_1(s^{-1})}\,ds+\frac{q^2(q-1)(q^2-1)}{3(q^3-1)}\left(f(q^{-1},q^{-1})+f(\omega q^{-1},\omega^{-1}q^{-1})+f(\omega^{-1}q^{-1},\omega q^{-1})\right).
$$
Therefore
\begin{align*}
\Tr(h)&=\frac{1}{q^3}\int_{\TT^2}\frac{f(t)}{|c(t)|^2}\,dt+\frac{(q-1)^2}{q^2(q^2-1)}\int_{c\TT}\frac{f(q^{-\frac{1}{2}}s,s^2)+f(s^2,q^{-\frac{1}{2}}s)+f(q^{-\frac{1}{2}}s,q^{-\frac{1}{2}}s^{-1})}{c_1(s)c_1(s^{-1})}\,ds\\
&\qquad+\frac{(q-1)^3}{3(q^3-1)}\left(f(q^{-1},q^{-1})+f(\omega q^{-1},\omega^{-1}q^{-1})+f(\omega^{-1}q^{-1},\omega q^{-1})\right).
\end{align*}
By (\ref{eq:A2_char}) the numerator of the single integral is $\chi_s(h)$, and is therefore defined on~$\TT$ and so the contour of the single integral can be shifted to~$\TT$. Proposition~\ref{prop:prelim_main} deals with the double integral, and Lemma~\ref{lem:1dim} deals with the $3$ constant terms.
\end{proof}

\subsection{The $\tilde{C}_2(q_1,q_2)$ algebras with $L=Q$}\label{sect:C2Q}

The dual root system is $R^{\vee}=\pm\{\alpha_1^{\vee},\alpha_2^{\vee},\alpha_1^{\vee}+\alpha_2^{\vee},\alpha_1^{\vee}+2\alpha_2^{\vee}\}$. Writing $x_1=x^{\alpha_1^{\vee}}$ and $x_2=x^{\alpha_2^{\vee}}$, the Hecke algebra has generators $T_1,T_2,x_1,x_2$ with relations
\begin{align*}
T_1^2&=1+(q_1^{\frac{1}{2}}-q_1^{-\frac{1}{2}})T_1& T_1x_1&=x_1^{-1}T_1+(q_1^{\frac{1}{2}}-q_1^{-\frac{1}{2}})(1+x_1) & T_1T_2T_1T_2&=T_2T_1T_2T_1\\
T_2^2&=1+(q_2^{\frac{1}{2}}-q_2^{-\frac{1}{2}})T_2 & T_2x_2&=x_2^{-1}T_2+(q_2^{\frac{1}{2}}-q_2^{-\frac{1}{2}})(1+x_2)& x_1x_2&=x_2x_1\\
T_1x_2&=x_1x_2T_1^{-1} & T_2x_1&=x_1x_2^2T_2^{-1}-(q_2^{\frac{1}{2}}-q_2^{-\frac{1}{2}})x_1x_2.
\end{align*}

Let $\pi_t=\mathrm{Ind}_{\CC[Q]}^{\scH}(\CC v_t)$ be the principal series representation of $\scH$ with central character $t=(t_1,t_2)\in(\CC^{\times})^2$, where $\CC v_t$ is the $1$-dimensional representation of $\CC[Q]$ with $x_1\cdot v_t=t_1v_t$ and $x_2\cdot v_t=t_2v_t$. 

Let $\scH_1$ be the subalgebra generated by $T_1,x_1,x_2$ and let $\scH_2$ be the subalgebra generated by $T_2,x_1,x_2$. Let $s\in\CC^{\times}$, and let $\pi_s^{1}=\mathrm{Ind}_{\scH_1}^{\scH}(\CC u_s^1)$ and $\pi_s^2=\mathrm{Ind}_{\scH_2}^{\scH}(\CC u_s^2)$ be the $4$-dimensional representations induced from the $1$-dimensional representation $\CC u_s^1$ of $\scH_1$ and the $1$-dimensional representation $\CC u_s^2$ of $\scH_2$
given by
\begin{align*}
T_1\cdot u_s^1&=-q_1^{-\frac{1}{2}}u_s^1 & x_1\cdot u_s^1&=q_1^{-1}u_s^1 & x_2\cdot u_s^1&=q_1^{\frac{1}{2}}su_s^1\\
T_2\cdot u_s^2&=-q_2^{-\frac{1}{2}}u_s^2 & x_1\cdot u_s^2&=q_2su_s^2 & x_2\cdot u_s^2&=q_2^{-1}u_s^2.
\end{align*}

Let $\pi^j$ ($j=3,4,5,6,7$) be the $1$-dimensional representations of $\scH$ with
\begin{align*}
\pi^3(T_1)&=-q_1^{-\frac{1}{2}} & \pi^3(T_2)&=-q_2^{-\frac{1}{2}} & \pi^3(x_1)&=q_1^{-1} & \pi^3(x_2)&=q_2^{-1}\\
\pi^4(T_1)&=-q_1^{-\frac{1}{2}} & \pi^4(T_2)&=-q_2^{-\frac{1}{2}} & \pi^4(x_1)&=q_1^{-1} & \pi^4(x_2)&=-1\\
\pi^5(T_1)&=-q_1^{-\frac{1}{2}} & \pi^5(T_2)&=q_2^{\frac{1}{2}} & \pi^5(x_1)&=q_1^{-1} & \pi^5(x_2)&=-1\\
\pi^6(T_1)&=q_1^{\frac{1}{2}} & \pi^6(T_2)&=-q_2^{-\frac{1}{2}} & \pi^6(x_1)&= q_1 & \pi^6(x_2)&= q_2^{-1}\\
\pi^7(T_1)&=-q_1^{-\frac{1}{2}} & \pi^7(T_2)&=q_2^{\frac{1}{2}} & \pi^7(x_1)&=q_1^{-1} & \pi^7(x_2)&=q_2.
\end{align*}

Suppose that $q_1\neq q_2$ and $q_1\neq q_2^2$. Let $\pi^8=M_J(t)$ be the representation with
\begin{align*}
(t^{\alpha_1^{\vee}},t^{\alpha_2^{\vee}})&=(q_1^{-1},q_2)&J^{\vee}&=\{\alpha_2^{\vee}\}&F_J(t)&=\{s_2,s_1s_2,s_2s_1s_2\}
\end{align*}
(since $q_1\neq q_2$ and $q_1\neq q_2^2$ we compute $N(t)^{\vee}=\{\alpha_1^{\vee},\alpha_2^{\vee}\}$ and $D(t)^{\vee}=\emptyset$). The matrices for $\pi^8$ are given in Example~3 of Section~\ref{sect:examples}.

Let $\chi_t$, $\chi_s^1$, $\chi_s^2$, and $\chi^j$ be the characters of $\pi_t$, $\pi_s^1$, $\pi_s^2$, and $\pi^j$ respectively ($j=3,\ldots,8$).

\begin{lemma}\label{lem:C2Q}
Let $t,u\in\Hom(Q,\CC^{\times})$ be $(t^{\alpha_1^{\vee}},t^{\alpha_2^{\vee}})=(q_1^{-1},q_1^{\frac{1}{2}}s)$ and $(u^{\alpha_1^{\vee}},u^{\alpha_2^{\vee}})=(q_2s,q_2^{-1})$ where $s\in\CC^{\times}$. For all $h\in\scH$ and all $s\in\CC^{\times}$ we have
\begin{align}
\label{eq:extend}\chi_s^1(h)&=f_{t}(h)+f_{s_2t}(h)+f_{s_1s_2t}(h)+f_{s_2s_1s_2t}(h)\\
\label{eq:extend2}\chi_s^2(h)&=f_u(h)+f_{s_1u}(h)+f_{s_2s_1u}(h)+f_{s_1s_2s_1u}(h).
\end{align}
\end{lemma}

\begin{proof} Let us prove (\ref{eq:extend}) ((\ref{eq:extend2}) is similar). Suppose that $s\in\CC^{\times}$ is not one of the isolated points of $\CC^{\times}$ which give $t^{\alpha^{\vee}}=1$ for some $\alpha\in R_0^+$. Then $\pi_s^1$ is irreducible (for example it can be constructed using Theorem~\ref{thm:repmain} in these cases) and each $f_{vt}(h)$ is defined (for $v\in W_0$ and $h\in\scH$). Moreover $\pi_s^1$ has basis $\{1\otimes u_s^1,\tau_2\otimes u_s^1,\tau_1\tau_2\otimes u_s^1,\tau_2\tau_1\tau_2\otimes u_s^1\}$ (this is proved in a similar way to the corresponding statement in the proof of Proposition~\ref{prop:prelim_main}).

The diagonal entries of each matrix $\pi_s^1(\tau_w)$ relative to this basis are~$0$. This is easily seen once it is observed that $\tau_1\otimes u_s^1=0$ (which can be seen by direct calculation, or by (\ref{eq:tausquared})). Since 
$$\pi_s^1(x^{\lambda})=\mathrm{diag}(t^{\lambda},(s_2t)^{\lambda},(s_1s_2t)^{\lambda},(s_2s_1s_2t)^{\lambda})\qquad\textrm{for all $\lambda\in Q$}$$
it follows that
$$
\chi_s^1(\tau_w x^{\lambda})=\delta_{w,1}\big(t^{\lambda}+(s_2t)^{\lambda}+(s_1s_2t)^{\lambda}+(s_2s_1s_2t)^{\lambda}\big)\quad\textrm{for all $w\in W_0$ and $\lambda\in Q$}.
$$ 
Thus Lemma~\ref{lem:milage} gives (\ref{eq:extend}) provided $s$ is not one of the isolated points of $\CC^{\times}$ that gives $t^{\alpha^{\vee}}= 1$ for some $\alpha\in R_0^+$. But by construction, $\chi_s^1(h)$ is a polynomial in $s$ and $s^{-1}$ (for fixed $h\in\scH$) and the right hand side of (\ref{eq:extend}) is a rational function in $s$. Hence the result.
\end{proof}

\begin{thm}\label{thm:C_2Q} For all $h\in\scH$ we have
\begin{align*}
\Tr(h)&=\frac{1}{8q_1^2q_2^2}\iint_{\TT^2}\frac{\chi_t(h)}{|c(t)|^2}\,dt+\frac{q_1-1}{2q_1q_2^2(q_1+1)}\int_{\TT}\frac{\chi_s^{1}(h)}{|c_1(s)|^2}\,ds+\frac{q_2-1}{2q_1^2q_2(q_2+1)}\int_{\TT}\frac{\chi_s^{2}(h)}{|c_2(s)|^2}\,ds\\
&\quad+A\chi^{3}(h)+B\left(\chi^{4}(h)+\chi^{5}(h)\right)+|C|\times\begin{cases}
\chi^6(h)&\textrm{if $q_1<q_2$}\\
\chi^8(h)&\textrm{if $q_2<q_1<q_2^2$}\\
\chi^7(h)&\textrm{if $q_2^2<q_1$},
\end{cases}
\end{align*}
where $c(t)$, $c_1(s)$, $c_2(s)$, $A$, $B$, and $C$ are as in Appendix~\ref{app:C2Q}. If $q_1=q_2$ or $q_1=q_2^2$ then the final term in the Plancherel Theorem is~$0$.
\end{thm}

\begin{proof}
The trace functional is given by
\begin{align}\label{eq:TrC2}
\Tr(h)=\frac{1}{q_1^2q_2^2}\int_{q_1^{-1}a\TT}\int_{q_2^{-1}b\TT}\frac{f(t)}{c(t)c(t^{-1})}\,dt_2dt_1,
\end{align}
where $0<a,b<1$ and where $f(t)=f_t(h)$. Choose $a$ with $a<q_1q_2^{-1}$. 

\smallskip

\noindent\textit{Step 1: Shifting the $t_2$-contour}. Let $0<c<1$ with $c^2>q_1^{-1}$, $c>q_2^{-1}$, $c>q_1q_2^{-1}$ (if $q_1<q_2$) and $c>q_1^{-1}q_2$ (if $q_2<q_1$). We will shift the $t_2$-contour from $q_2^{-1}b\TT$ to $c\TT$. The integrand has exactly one $t_2$-pole between these contours, at $t_2=q_2^{-1}$. Thus
$$
\Tr(h)=\frac{1}{q_1^2q_2^2}\int_{q_1^{-1}a\TT}\int_{c\TT}\frac{f(t)}{c(t)c(t^{-1})}\,dt_2dt_1+I_1,
\quad\textrm{where}\quad
I_1=-\frac{1}{q_1^2q_2^2}\int_{q_1^{-1}a\TT}\Res_{t_2=q_2^{-1}}\frac{f(t)}{c(t)c(t^{-1})}\,dt.
$$

\noindent\textit{Step 2: Shifting the $t_1$-contour.} Interchange the order of integration in the double integral. We will shift the $t_1$-contour from $q_1^{-1}a\TT$ to $\TT$. By the conditions on $a$ and $c$ the $t_1$-poles of the integrand between these contours are at $t_1=q_1^{-1}$, $t_1=q_1^{-1}t_2^{-2}$, and $t_1=q_2^{-1}t_2^{-1}$. Therefore
$$
\Tr(h)=\frac{1}{q_1^2q_2^2}\int_{c\TT}\int_{\TT}\frac{f(t)}{c(t)c(t^{-1})}\,dt_1dt_2+I_1+I_2+I_3+I_4,
$$
where
$$
I_j=-\frac{1}{q_1^2q_2^2}\int_{c\TT}\Res_{t_1=z_j}\frac{f(t)}{c(t)c(t^{-1})}\,dt_2\qquad\textrm{for $j=2,3,4$},
$$
with $z_2=q_1^{-1}$, $z_3=q_1^{-1}t_2^{-2}$, and $z_4=q_2^{-1}t_2^{-1}$. In the double integral we may now revert back to the original order of integration, and shift the $t_2$-contour to $\TT$ without encountering any poles.

\smallskip

\noindent\textit{Step 3: Shifting the contours in the integrals $I_j$.} Straightforward calculations give
\begin{align*}
I_1&=\frac{(q_2-1)^2}{q_1^2q_2(q_2^2-1)}\int_{q_1^{-1}q_2^{-1}a\TT}\frac{f(q_2s,q_2^{-1})}{c_2(s)c_2(s^{-1})}\,ds &
I_2&=\frac{(q_1-1)^2}{q_1q_2^2(q_1^2-1)}\int_{q_1^{-\frac{1}{2}}c\TT}\frac{f(q_1^{-1},q_1^{\frac{1}{2}}s)}{c_1(s)c_1(s^{-1})}\,ds\\
I_3&=\frac{(q_1-1)^2}{q_1q_2^2(q_1^2-1)}\int_{q_1^{\frac{1}{2}}c\TT}\frac{f(s^{-2},q_1^{-\frac{1}{2}}s)}{c_1(s)c_1(s^{-1})}\,ds &
I_4&=\frac{(q_2-1)^2}{q_1^2q_2(q_2^2-1)}\int_{c\TT}\frac{f(q_2^{-1}s^{-1},s)}{c_2(s)c_2(s^{-1})}\,ds,
\end{align*}
where we have set $s=q_2^{-1}t_1$ in $I_1$, $s=q_1^{-\frac{1}{2}}t_2$ in $I_2$, $s=q_1^{\frac{1}{2}}t_2$ in $I_3$, and $s=t_2$ in $I_4$. 

We now shift each contour to $\TT$. As in the $\tilde{A}_2$ case we need to be a little careful with possible singularities of $f(t)$. Thus we write $f(t)=g(t)/d(t)$. Then the integrands of $I_1,I_2,I_3$ and $I_4$ are
\begin{align*}
\frac{f(q_2s,q_2^{-1})}{c_2(s)c_2(s^{-1})}&=\frac{q_2s(1-s)g(q_2s,q_2^{-1})}{(q_2-1)n_2(s)n_2(s^{-1})}&
\frac{f(q_1^{-1},q_1^{\frac{1}{2}}s)}{c_1(s)c_1(s^{-1})}&=\frac{q_1^{\frac{1}{2}}s(1-s^2)g(q_1^{-1},q_1^{\frac{1}{2}}s)}{(q_1-1)n_1(s)n_1(s^{-1})}\\
\frac{f(s^{-2},q_1^{-\frac{1}{2}}s)}{c_1(s)c_1(s^{-1})}&=\frac{(1-s^{-2})g(s^{-2},q_1^{-\frac{1}{2}}s)}{(1-q_1)n_1(s)n_1(s^{-1})}&
\frac{f(q_2^{-1}s^{-1},s)}{c_2(s)c_2(s^{-1})}&=\frac{(1-s)g(q_2^{-1}s^{-1},s)}{(1-q_2)n_2(s)n_2(s^{-1})},
\end{align*}
where $n_1(s)$ and $n_2(s)$ are the numerators of $c_1(s)$ and $c_2(s)$. Each integrand is nonsingular on~$\TT$ (with removable singularities in the cases $q_1=q_2$ or $q_1=q_2^2$).

The poles of the integrand of $I_1$ which lie between the contours $q_1^{-1}q_2^{-1}a\TT$ and $\TT$ are at $s=q_1^{-1}q_2^{-1}$, $s=q_1^{-1}q_2$ (if $q_2<q_1$) and at $s=q_1q_2^{-1}$ (if $q_1<q_2$). Calculating residues gives
$$
I_1=\frac{(q_2-1)^2}{q_1^2q_2(q_2^2-1)}\int_{\TT}\frac{f(q_2s,q_2^{-1})}{|c_2(s)|^2}\,ds+Af(q_1^{-1},q_2^{-1})+C\times\begin{cases}f(q_1^{-1}q_2^2,q_2^{-1})&\textrm{if $q_2<q_1$}\\
-f(q_1,q_2^{-1})&\textrm{if $q_1<q_2$}.
\end{cases}
$$

The poles of the integrand of $I_2$ which lie between the contours $q_1^{-\frac{1}{2}}b\TT$ and $\TT$ are at $s=-q_1^{-\frac{1}{2}}$, $s=q_1^{\frac{1}{2}}q_2^{-1}$ (if $q_2<q_1<q_2^2$), and $s=q_1^{-\frac{1}{2}}q_2$ (if $q_2^2<q_1$). Calculating residues gives
$$
I_2=\frac{(q_1-1)^2}{q_1q_2^2(q_1^2-1)}\int_{\TT}\frac{f(q_1^{-1},q_1^{\frac{1}{2}}s)}{c_1(s)c_1(s^{-1})}\,ds+2Bf(q_1^{-1},-1)+C\times\begin{cases}
f(q_1^{-1},q_1q_2^{-1})&\textrm{if $q_2<q_1<q_2^2$}\\
-f(q_1^{-1},q_2)&\textrm{if $q_2^2<q_1$}.
\end{cases}
$$

The poles of the integrand of $I_3$ which lie between the contours $q_1^{\frac{1}{2}}c\TT$ and $\TT$ are at $s=q_1^{-\frac{1}{2}}q_2$ (if $q_2<q_1<q_2^2$) and $s=q_1^{\frac{1}{2}}q_2^{-1}$ (if $q_2^2<q_1$). Noting that $\TT$ is inside $q_1^{\frac{1}{2}}c\TT$ gives
$$
I_3=\frac{(q_1-1)^2}{q_1q_2^2(q_1^2-1)}\int_{\TT}\frac{f(s^{-2},q_1^{-\frac{1}{2}}s)}{c_1(s)c_1(s^{-1})}\,ds+C\times\begin{cases}
f(q_1q_2^{-2},q_1^{-1}q_2)&\textrm{if $q_2<q_1<q_2^2$}\\
-f(q_1^{-1}q_2^{2},q_2^{-1})&\textrm{if $q_2^2<q_1$}.
\end{cases}
$$

The integrand of $I_4$ has no poles between $c\TT$ and $\TT$. Therefore
\begin{align*}
\Tr(h)&=\frac{1}{q_1^2q_2^2}\iint_{\TT^2}\frac{f(t)}{|c(t)|^2}\,dt+\frac{(q_1-1)^2}{q_1q_2^2(q_1^2-1)}\int_{\TT}\frac{f(q_1^{-1},q_1^{\frac{1}{2}}s)+f(s^{-2},q_1^{-\frac{1}{2}}s)}{|c_1(s)|^2}\,ds\\
&\quad+\frac{(q_2-1)^2}{q_1^2q_2(q_2^2-1)}\int_{\TT}\frac{f(q_2s,q_2^{-1})+f(q_2^{-1}s^{-1},s)}{|c_2(s)|^2}\,ds+Af(q_1^{-1},q_2^{-1})+2Bf(q_1^{-1},-1)\\
&\quad+|C|\times\begin{cases}
f(q_1,q_2^{-1})&\textrm{if $q_1<q_2$}\\
f(q_1^{-1}q_2^2,q_2^{-1})+f(q_1^{-1},q_1q_2^{-1})+f(q_1q_2^{-2},q_1^{-1}q_2)&\textrm{if $q_2<q_1<q_2^2$}\\
f(q_1^{-1},q_2)&\textrm{if $q_2^2<q_1$}.
\end{cases}
\end{align*}

\noindent\textit{Step 4: Matching with the representations.} By Proposition~\ref{prop:prelim_main} the double integral in the above formula is
$$
\int_{\TT^2}\frac{f(t)}{|c(t)|^2}\,dt=\frac{1}{8}\int_{\TT^2}\frac{\chi_t(h)}{|c(t)|^2}\,dt.
$$
The first single integral is
$$
\frac{1}{2}\int_{\TT}\frac{f(q_1^{-1},q_1^{\frac{1}{2}}s)+f(s^{-2},q_1^{-\frac{1}{2}}s)+f(q_1^{-1},q_1^{\frac{1}{2}}s^{-1})+f(s^2,q_1^{-\frac{1}{2}}s^{-1})}{|c_1(s)|^2}\,ds=\frac{1}{2}\int_{\TT}\frac{\chi_s^1(h)}{|c_1(s)|^2}\,ds,
$$
where we have used Lemma~\ref{lem:C2Q}. A similar analysis applies to the second single integral. Using Lemma~\ref{lem:1dim} we have $f(q_1^{-1},q_2^{-1})=\chi^3(h)$ and $2f(q_1^{-1},-1)=\chi^4(h)+\chi^5(h)$. Furthermore, for parameters $q_1<q_2$ the central character $(t_1,t_2)=(q_1,q_2^{-1})$ is regular (for $t^{\alpha_1^{\vee}+\alpha_2^{\vee}}=q_1q_2^{-1}<1$ and $t^{\alpha_1^{\vee}+2\alpha_2^{\vee}}=q_1q_2^{-2}<1$), and so Lemma~\ref{lem:1dim} gives $f(q_1,q_2^{-1})=\chi^6(h)$. Similarly we have $f(q_1^{-1},q_2)=\chi^7(h)$ for parameters $q_2^2<q_1$. Finally by Proposition~\ref{prop:calc} we have 
$$f(q_1^{-1}q_2^2,q_2^{-1})+f(q_1^{-1},q_1q_2^{-1})+f(q_1q_2^{-2},q_1^{-1}q_2)=\chi^8(h)$$
for all parameters in the range $q_2<q_1<q_2^2$ (as the central character is regular).
\end{proof}

\subsection{The $\tilde{C}_2(q_1,q_2)$ algebras with $L=P$}\label{sect:C2P}

The root system is as in Section~\ref{sect:C2Q}, and the fundamental coweights are given by $\omega_1=\alpha_1^{\vee}+\alpha_2^{\vee}$ and $\omega_2=\frac{1}{2}\alpha_1^{\vee}+\alpha_2^{\vee}$. Writing $x_1=x^{\omega_1}$ and $x_2=x^{\omega_2}$, the Hecke algebra has presentation given by generators $T_1,T_2,x_1,x_2$ with relations
\begin{align*}
T_1^2&=1+(q_1^{\frac{1}{2}}-q_1^{-\frac{1}{2}})T_1& T_1x_1&=x_1^{-1}x_2^2T_1+(q_1^{\frac{1}{2}}-q_1^{-\frac{1}{2}})x_1 & T_1T_2T_1T_2&=T_2T_1T_2T_1\\
T_2^2&=1+(q_2^{\frac{1}{2}}-q_2^{-\frac{1}{2}})T_2 & T_2x_2&=x_1x_2^{-1}T_2+(q_2^{\frac{1}{2}}-q_2^{-\frac{1}{2}})x_2& x_1x_2&=x_2x_1\\
T_1x_2&=x_2T_1 & T_2x_1&=x_1T_2.
\end{align*}
The representation theory of $\scH$ is closely related to the representation theory of the Hecke algebra from Section~\ref{sect:C2Q}.

Let $\pi_t=\mathrm{Ind}_{\CC[P]}^{\scH}(\CC v_t)$ be the principal series representation of $\scH$ with central character $t=(t_1,t_2)\in(\CC^{\times})^2$, where $\CC v_t$ is the $1$-dimensional representation of $\CC[P]$ with $x_1\cdot v_t=t_1v_t$ and $x_2\cdot v_t=t_2v_t$. 

Let $\scH_1$ be the subalgebra generated by $T_1,x_1,x_2$. 
Let $s\in\CC^{\times}$, and let $\pi_s^{\pm}=\mathrm{Ind}_{\scH_1}^{\scH}(\CC u_s^{\pm})$ be the $4$-dimensional representations of $\scH$ induced from the representations $\CC u_s^{\pm}$ of $\scH_1$ with
\begin{align*}
T_1\cdot u_s^{\pm}&=-q_1^{-\frac{1}{2}}u_s^{\pm}& x_1\cdot u_s^{\pm}&=q_1^{-\frac{1}{2}}su_s^{\pm}& x_2\cdot u_s^{\pm}&=\pm su_s^{\pm}.
\end{align*}
Let $\scH_2$ be the subalgebra generated by $T_2,x_1,x_2$. Let $s\in\CC^{\times}$, and let $\pi_s^2=\mathrm{Ind}_{\scH_2}^{\scH}(\CC u_s^2)$ be the $4$-dimensional representation of $\scH$ induced from the representation $\CC u_s^2$ of $\scH_2$ with
\begin{align*}
T_2\cdot u_s^2&=-q_2^{-\frac{1}{2}}u_s^2& x_1\cdot u_s^2&=s^2u_s^2& x_2\cdot u_s^2&=q_2^{-\frac{1}{2}}su_s^2.
\end{align*}

Let $\pi^j_{\pm}$ ($j=3,4,5$) be the $1$-dimensional representations
\begin{align*}
\pi^3_{\pm}(T_1)&=-q_1^{-\frac{1}{2}}& \pi^3_{\pm}(T_2)&=-q_2^{-\frac{1}{2}} & \pi^3_{\pm}(x_1)&=q_1^{-1}q_2^{-1} & \pi^3_{\pm}(x_2)&=\pm q_1^{-\frac{1}{2}}q_2^{-1}\\
\pi^4_{\pm}(T_1)&=q_1^{\frac{1}{2}} & \pi^4_{\pm}(T_2)&=-q_2^{-\frac{1}{2}} & \pi^4_{\pm}(x_1)&=q_1q_2^{-1} & \pi^4_{\pm}(x_2)&=\pm q_1^{\frac{1}{2}}q_2^{-1}\\
\pi^5_{\pm}(T_1)&=-q_1^{-\frac{1}{2}} & \pi^5_{\pm}(T_2)&=q_2^{\frac{1}{2}} & \pi^5_{\pm}(x_1)&=q_1^{-1}q_2 & \pi^5_{\pm}(x_2)&=\pm q_1^{-\frac{1}{2}}q_2.
\end{align*}

Let $\pi^6=M_J(t)$ be the $2$-dimensional representation with
\begin{align*}
(t^{\omega_1},t^{\omega_2})&=(-q_1^{-1},q_1^{-\frac{1}{2}})&J^{\vee}&=\emptyset&F_J(t)&=\{1,s_2\}
\end{align*}
(we have $N(t)^{\vee}=\{\alpha_1^{\vee},\alpha_1^{\vee}+2\alpha_2^{\vee}\}$ and $D(t)^{\vee}=\emptyset$). Coincidentally, $\pi^6\cong\mathrm{Ind}_{\scH_Q}^{\scH}(\CC u)$ where $\scH_Q$ is the algebra from Section~\ref{sect:C2Q} and where $\CC u$ is the $1$-dimensional representation of $\scH_Q$ with 
$
T_1\cdot u=-q_1^{-1/2}u$, $T_2\cdot u=-q_2^{-1/2}u$, $x^{\alpha_1^{\vee}}\cdot u=q_1^{-1}u$, and $x^{\alpha_2^{\vee}}\cdot u=-u$. The matrices for $\pi^6$ are given in Example~1 of Section~\ref{sect:examples}.

Suppose that $q_1\neq q_2$ and $q_1\neq q_2^2$. Let $\pi^7_{\pm}=M_J(t_{\pm})$ be the $3$-dimensional representations with
\begin{align*}
(t_{\pm}^{\omega_1},t_{\pm}^{\omega_2})&=(q_1^{-1}q_2,\pm q_1^{-\frac{1}{2}}q_2)&J^{\vee}&=\{\alpha_2^{\vee}\}&F_J(t_{\pm})&=\{s_2,s_1s_2,s_2s_1s_2\}
\end{align*}
(we calculate $N(t_{\pm})^{\vee}=\{\alpha_1^{\vee},\alpha_2^{\vee}\}$ and $D(t)^{\vee}=\emptyset$).

\begin{thm} 
For all $h\in\scH$ we have
\begin{align*}
\Tr(h)&=\frac{1}{8q_1^2q_2^2}\iint_{\TT^2}\frac{\chi_t(h)}{|c(t)|^2}\,dt+\frac{(q_1-1)^2}{4q_1q_2^2(q_1^2-1)}\int_{\TT}\frac{\chi_s^+(h)+\chi_s^-(h)}{|c_1(s)|^2}\,ds+\frac{(q_2-1)^2}{2q_1^2q_2(q_2^2-1)}\int_{\TT}\frac{\chi_s^2(h)}{|c_2(s)|^2}\,ds\\
&\quad+\frac{A}{2}\left(\chi^3_+(h)+\chi^3_-(h)\right)+B\chi^6(h)+\frac{|C|}{2}\times\begin{cases}
\chi^4_+(h)+\chi^4_-(h)&\textrm{if $q_1<q_2$}\\
\chi^7_+(h)+\chi^7_-(h)&\textrm{if $q_2<q_1<q_2^2$}\\
\chi^5_+(h)+\chi^5_-(h)&\textrm{if $q_2^2<q_1$}
\end{cases}
\end{align*}
where $c(t)$, $c_1(s)$, $c_2(s)$ are as in Appendix~\ref{app:C2P} and $A,B,C$ are as in Appendix~\ref{app:C2Q}. If $q_1=q_2$ or $q_1=q_2^2$ then the final term in the Plancherel Theorem is~$0$.
\end{thm}

\begin{proof} The series $G_t(h)$ converges for $|t^{\alpha_1^{\vee}}|<q_1^{-1}$ and $|t^{\alpha_2^{\vee}}|<q_2^{-1}$. Since $\alpha_1^{\vee}=2\omega_1-2\omega_2$ and $\alpha_2^{\vee}=-\omega_1+2\omega_2$ the series converges whenever $|t_1^2t_2^{-2}|<q_1^{-1}$ and $|t_1^{-1}t_2^2|<q_2^{-1}$. Thus, writing $|t_1|=q_1^{-1}q_2^{-1}a$ and $|t_2|=q_1^{-1/2}q_2^{-1}b$, the series converges for $b^2<a<b<1$, and so
$$
\Tr(h)=\frac{1}{q_1^2q_2^2}\int_{q_1^{-1}q_2^{-1}a\TT}\int_{q_1^{-1/2}q_2^{-1}b\TT}\frac{f(t)}{c(t)c(t^{-1})}\,dt_2dt_1.
$$
From here one can either perform the contour shifts as in the $L=Q$ case, or change variables $t_1=u_1u_2$ and $t_2^2=u_1u_2^2$ to transform the above integral into $\frac{1}{2}$ times the $L=Q$ integral (\ref{eq:TrC2}) with the numerator of its integrand replaced by $f'(u_1,u_2)=f(u_1u_2,u_1^{1/2}u_2)+f(u_1u_2,-u_1^{1/2}u_2)$.
We omit the details.
\end{proof}

\subsection{The $\tilde{G}_2(q_1,q_2)$ algebras with $L=Q$}\label{sect:G2}

The coroot system is
$
R^{\vee}=\pm\{\alpha_1^{\vee},2\alpha_1^{\vee}+3\alpha_2^{\vee},\alpha_1^{\vee}+3\alpha_2^{\vee},\alpha_2^{\vee},\alpha_1^{\vee}+2\alpha_2^{\vee},\alpha_1^{\vee}+\alpha_2^{\vee}\},
$
and the reflections $s_1$ and $s_2$ are given by $s_1(\alpha_2^{\vee})=\alpha_1^{\vee}+\alpha_2^{\vee}$ and $s_2(\alpha_1^{\vee})=\alpha_1^{\vee}+3\alpha_2^{\vee}$. Writing $x_1=x^{\alpha_1^{\vee}}$ and $x_2=x^{\alpha_2^{\vee}}$, the Hecke algebra $\scH$ has generators $T_1$, $T_2$, $x_1$ and $x_2$ with relations
\begin{align*}
T_1^2&=1+(q_1^{\frac{1}{2}}-q_1^{-\frac{1}{2}})T_1& T_1x_1&=x_1^{-1}T_1+(q_1^{\frac{1}{2}}-q_1^{-\frac{1}{2}})(1+x_1)\\
T_2^2&=1+(q_2^{\frac{1}{2}}-q_2^{-\frac{1}{2}})T_2 & T_2x_2&=x_2^{-1}T_2+(q_2^{\frac{1}{2}}-q_2^{-\frac{1}{2}})(1+x_2)\\
T_1T_2T_1T_2T_1T_2&=T_2T_1T_2T_1T_2T_1 & T_2x_1&=x_1x_2^3T_2^{-1}-(q_2^{\frac{1}{2}}-q_2^{-\frac{1}{2}})x_1x_2(1+x_2)\\
 x_1x_2&=x_2x_1 &
 T_1x_2&=x_1x_2T_1^{-1}.
\end{align*}

Let $\pi_t=\mathrm{Ind}_{\CC[Q]}^{\scH}(\CC v_t)$ be the principal series representation of $\scH$ with central character $t=(t_1,t_2)\in(\CC^{\times})^2$, where $\CC v_t$ is the $1$-dimensional representation of $\CC[Q]$ with $x_1\cdot v_t=t_1v_t$ and $x_2\cdot v_t=t_2v_t$. 

Let $\scH_1$ be the subalgebra of $\scH$ generated by $T_1,x_1,x_2$, and let $\scH_2$ be the subalgebra generated by $T_2,x_1,x_2$. Let $s\in\CC^{\times}$, and let $\pi_s^1=\mathrm{Ind}_{\scH_1}^{\scH}(\CC u_s^1)$ and $\pi_s^2=\mathrm{Ind}_{\scH_2}^{\scH}(\CC u_s^2)$ be the $6$-dimensional representations induced from the $1$-dimensional representation $\CC u_s^1$ of $\scH_1$ and the $1$-dimensional representation $\CC u_s^2$ of $\scH_2$ given by 
\begin{align*}
T_1\cdot u_s^1&=-q_1^{-\frac{1}{2}}u_s^1 & x_1\cdot u_s^1&=q_1^{-1}u_s^1 & x_2\cdot u_s^1&=q_1^{\frac{1}{2}}su_s^1\\
T_2\cdot u_s^2&=-q_2^{-\frac{1}{2}}u_s^2 & x_1\cdot u_s^2&=q_2^{\frac{3}{2}}su_s^2 & x_2\cdot u_s^2&=q_2^{-1}u_s^2.
\end{align*}

Let $\pi^3,\pi^4$ and $\pi^5$ be the $1$-dimensional representations of $\scH$ with
\begin{align*}
\pi^3(T_1)&=-q_1^{-\frac{1}{2}}& \pi^3(T_2)&=-q_2^{-\frac{1}{2}}& \pi^3(x_1)&=q_1^{-1}& \pi^3(x_2)&=q_2^{-1}\\
\pi^4(T_1)&=q_1^{\frac{1}{2}}&\pi^4(T_2)&=-q_2^{-\frac{1}{2}}&\pi^4(x_1)&=q_1&\pi^4(x_2)&=q_2^{-1}\\
\pi^5(T_1)&=-q_1^{-\frac{1}{2}}&\pi^5(T_2)&=q_2^{\frac{1}{2}}&\pi^5(x_1)&=q_1^{-1}&\pi^5(x_2)&=q_2.
\end{align*}

Suppose that $q_1\neq q_2$, $q_1\neq q_2^2$, $q_1^2\neq q_2^3$, $q_1\neq q_2^3$. Let $\pi^6=M_J(t)$ the the $5$-dimensional representation with 
\begin{align*}
(t^{\alpha_1^{\vee}},t^{\alpha_2^{\vee}})&=(q_1^{-1},q_2)&J^{\vee}&=\{\alpha_2^{\vee}\}&F_J(t)&=\{s_2,s_1s_2,s_2s_1s_2,s_1s_2s_1s_2,s_2s_1s_2s_1s_2\}.
\end{align*}

Let $\pi^7_{\pm}=M_J(t_{\pm})$ be the $3$-dimensional representations with
\begin{align*}
(t_{\pm}^{\alpha_1^{\vee}},t_{\pm}^{\alpha_2^{\vee}})&=(q_1,\pm q_1^{-\frac{1}{2}}q_2^{\frac{1}{2}})&J^{\vee}&=\{\alpha_1^{\vee}+2\alpha_2^{\vee}\}&F_J(t_{\pm})&=\{s_2s_1s_2s_1,s_1s_2s_1s_2s_1,s_2s_1s_2s_1s_2s_1\},
\end{align*}
where we assume that $q_1\neq q_2$ and $q_1\neq q_2^3$ for $M_J(t_+)$. When $q_1=q_2$ or $q_1=q_2^3$ we define $\pi_+^7$ differently, as explained in Example~2 of Section~\ref{sect:examples}.

Let $\pi^8=M_J(t)$ be the $2$-dimensional representation with
\begin{align*}
(t^{\alpha_1^{\vee}},t^{\alpha_2^{\vee}})&=(q_1,\omega)&J^{\vee}&=\{\alpha_1^{\vee}+3\alpha_2^{\vee}\}&F_J(t)&=\{s_1s_2s_1s_2s_1,s_2s_1s_2s_1s_2s_1\}.
\end{align*}

Let $\chi_t$, $\chi_s^1$, $\chi_s^2$, $\chi^3$, $\chi^4$, $\chi^5$, $\chi^6$, $\chi^7_{\pm}$ and $\chi^8$ be the characters of the above representations.

\begin{thm} For all $h\in\scH$ we have
\begin{align*}
\Tr(h)&=\frac{1}{12q_1^3q_2^3}\iint_{\TT^2}\frac{\chi_t(h)}{|c(t)|^2}\,dt+\frac{(q_1-1)^2}{2q_1q_2^3(q_1^2-1)}\int_{\TT}\frac{\chi_s^{1}(h)}{|c_1(s)|^2}\,ds+\frac{(q_2-1)^2}{2q_1^3q_2^2(q_2^2-1)}\int_{\TT}\frac{\chi_s^{2}(h)}{|c_2(s)|^2}\,ds\\
&\qquad+A\chi^3(h)+B_+\chi^7_+(h)+B_-\chi^7_-(h)+C\chi^8(h)+|D|\times\begin{cases}
\chi^4(h)&\textrm{if $q_1<q_2^{3/2}$}\\
\chi^6(h)&\textrm{if $q_2^{3/2}<q_1<q_2^2$}\\
\chi^5(h)&\textrm{if $q_2^2<q_1$}
\end{cases}
\end{align*}
where $c(t),c_1(s),c_2(s),A,B_{\pm},C,D$ are as in Appendix~\ref{app:G2Q}. If $q_1=q_2^{3/2}$ or $q_1=q_2^2$ then the final term in the Plancherel Theorem is~$0$. 
\end{thm}

\begin{proof}
Writing $f(t)=f_t(h)$, the trace functional is given by
$$
\Tr(h)=\frac{1}{q_1^3q_2^3}\int_{q_1^{-1}a\TT}\int_{q_2^{-1}b\TT}\frac{f(t)}{c(t)c(t^{-1})}\,dt_2dt_1
$$
with $0<a,b<1$. Choose $0<a,b<1$ both very close to $0$. Let $0<c<1$ be very close to~$1$.
Consider the inner integral. The integrand has exactly one $t_2$-pole between the contours $q_2^{-1}b\TT$ and $c\TT$, at $t_2=q_2^{-1}$. Thus we can shift the $t_2$-contour to $c\TT$ at the cost of including this residue contribution. Now interchange the order of integration in the double integral. Since $|t_2|=c$ we see that the $t_1$-poles of the integrand between the contours $q_1^{-1}a\TT$ and $\TT$ are at the points where $t_1=q_1^{-1}$, $t_1=q_1^{-1}t_2^{-3}$, $t_1=q_2^{-1}t_2^{-2}$, $t_1=q_2^{-1}t_2^{-1}$ and $t_1^2=q_1^{-1}t_2^{-3}$. After shifting the $t_1$-contour to $\TT$ we interchange the order of integration again, and since there are no $t_2$-poles between $c\TT$ and $\TT$ we shift the $t_2$-contour to $\TT$. Thus
$$
\Tr(h)=\frac{1}{q_1^3q_2^3}\iint_{\TT^2}\frac{f(t)}{|c(t)|^2}\,dt+I_1+I_2+I_3+I_4+I_5+I_6^++I_6^-
$$
where
\begin{align*}
I_1&=-\frac{1}{q_1^3q_2^3}\int_{q_1^{-1}a\TT}\Res_{t_2=z_1}\frac{f(t)}{c(t)c(t^{-1})}\,dt_1&I_6^{\pm}&=-\frac{1}{q_1^3q_2^3}\int_{c\TT}\Res_{t_1=\pm z_6}\frac{f(t)}{c(t)c(t^{-1})}dt_2\\
I_j&=-\frac{1}{q_1^3q_2^3}\int_{c\TT}\Res_{t_1=z_j}\frac{f(t)}{c(t)c(t^{-1})}dt_2& (j&=2,3,4,5),
\end{align*}
where $z_1=q_2^{-1}$, $z_2=q_1^{-1}$, $z_3=q_1^{-1}t_2^{-3}$, $z_4=q_2^{-1}t_2^{-2}$, $z_5=q_2^{-1}t_2^{-1}$, and $z_6=q_1^{-\frac{1}{2}}t_2^{-\frac{3}{2}}$

Use $\frac{1}{2}\int_{r\TT}(f(t^{1/2})+f(-t^{1/2}))dt=\int_{r^{1/2}\TT}f(t)dt$ to write $I_6^++I_6^-=I_6$ as a single integral. Straightforward calculations give
\begin{align*}
I_1&=\frac{(q_2-1)^2}{q_1^3q_2^2(q_2^2-1)}\int_{q_1^{-1}q_2^{-\frac{3}{2}}a\TT}\frac{f(q_2^{\frac{3}{2}}s,q_2^{-1})}{c_2(s)c_2(s^{-1})}\,ds 
& 
I_2&=\frac{(q_1-1)^2}{q_1q_2^3(q_1^2-1)}\int_{q_1^{-\frac{1}{2}}c\TT}\frac{f(q_1^{-1},q_1^{\frac{1}{2}}s)}{c_1(s)c_1(s^{-1})}\,ds\\
I_3&=\frac{(q_1-1)^2}{q_1q_2^3(q_1^2-1)}\int_{q_1^{\frac{1}{2}}c\TT}\frac{f(q_1^{\frac{1}{2}}s^{-3},q_1^{-\frac{1}{2}}s)}{c_1(s)c_1(s^{-1})}\,ds
&
I_4&=\frac{(q_2-1)^2}{q_1^3q_2^2(q_2^2-1)}\int_{q_2^{\frac{1}{2}}c\TT}\frac{f(s^{-2},q_2^{-\frac{1}{2}}s)}{c_2(s)c_2(s^{-1})}\,ds\\
I_5&=\frac{(q_2-1)^2}{q_1^3q_2^2(q_2^2-1)}\int_{q_2^{-\frac{1}{2}}c\TT}\frac{f(q_2^{-\frac{3}{2}}s^{-1},q_2^{\frac{1}{2}}s)}{c_2(s)c_2(s^{-1})}\,ds
&
I_6&=\frac{(q_1-1)^2}{q_1q_2^3(q_1^2-1)}\int_{c^{\frac{1}{2}}\TT}\frac{f(q_1^{-\frac{1}{2}}s^{-3},s^2)}{c_1(s)c_1(s^{-1})}\,ds
\end{align*}
where we have put $s=q_2^{-\frac{3}{2}}t_1,q_1^{-\frac{1}{2}}t_2,q_1^{\frac{1}{2}}t_2,q_2^{\frac{1}{2}}t_2,q_2^{-\frac{1}{2}}t_2,t_2$ in $I_1,I_2,I_3,I_4,I_5,I_6$ respectively. 

One now shifts each contour to $\TT$. As we discuss below, some complications arise when $q_1=q_2$ or $q_1=q_2^3$, and so suppose for now that $q_1\neq q_2$ and $q_1\neq q_2^3$. As in the $\tilde{C}_2$, $L=Q$, case the integrands of $I_1,\ldots,I_6$ are all nonsingular on~$\TT$. Moreover, assuming that $q_1\neq q_2$ and $q_1\neq q_2^3$ all singularities are simple poles, and at the special values $q_1^2=q_2^3$ or $q_1=q_2^2$ there are some removable singularities. Write $I_1^u,\ldots,I_6^u$ for the integrals over the contour~$\TT$. A lengthy analysis (using the fact that $a$ is close to $0$ and $c$ is close to $1$) gives
\begin{align*}
\Tr(h)&=\frac{1}{q_1^3q_2^3}\iint_{\TT^2}\frac{f(t)}{|c(t)|^2}\,dt+I_1^u+\cdots+I_6^u+Af(q_1^{-1},q_2^{-1})+B_+\sigma_++B_-\sigma_-\\
&\qquad+C\left(f(q_1^{-1},\omega)+f(q_1^{-1},\omega^{-1})\right)+|D|\times\begin{cases}
f(q_1,q_2^{-1})&\textrm{if $q_1<q_2^{3/2}$}\\
\sigma&\textrm{if $q_2^{3/2}<q_1<q_2^2$}\\
f(q_1^{-1},q_2)&\textrm{if $q_2^2<q_1$},
\end{cases}
\end{align*}
where $\sigma_{\pm}=f(\pm q_1^{-1/2}q_2^{3/2},q_2^{-1})+f(q_1^{-1},\pm q_1^{1/2}q_2^{-1/2})+f(\pm q_1^{1/2}q_2^{-3/2},\pm q_1^{-1/2}q_2^{1/2})$ and 
\begin{align*}
\sigma&=f(q_1^{-1}q_2^3,q_2^{-1})+f(q_1^{-1},q_1q_2^{-1})+f(q_1^2q_2^{-3},q_1^{-1}q_2)+f(q_1^{-2}q_2^3,q_1q_2^{-2})+f(q_1q_2^{-3},q_1^{-1}q_2^2).
\end{align*}
As in the previous sections it is easy to show that
$$
I_1^u+\cdots+I_6^u=\frac{(q_1-1)^2}{2q_1q_2^3(q_1^2-1)}\int_{\TT}\frac{\chi_s^{1}(h)}{|c_1(s)|^2}\,ds+\frac{(q_2-1)^2}{2q_1^3q_2^2(q_2^2-1)}\int_{\TT}\frac{\chi_s^{2}(h)}{|c_2(s)|^2}\,ds.
$$
Proposition~\ref{prop:calc} gives $f(q_1^{-1},\omega)+f(q_1^{-1},\omega^{-1})=\pi^8(h)$ and $\sigma=\pi^6(h)$ (note that $\pi^6$ only occurs for parameters $q_2^{3/2}<q_1<q_2^2$, and in this range $\pi^6$ is defined and has regular central character). We also have $\sigma_{\pm}=\pi_{\pm}^7(h)$. For $\pi_+^7$ it is important that $q_1\neq q_2$ and $q_1\neq q_2^3$, for otherwise $\pi_+^7$ does not have a regular central character and things become complicated (see below). Lemma~\ref{lem:1dim} gives $f(q_1^{-1},q_2^{-1})=\chi^3(h)$. Since we exclude $q_1=q_2$ the representation $\pi^4$ has regular central character for parameters $q_1<q_2^{3/2}$, and so Lemma~\ref{lem:1dim} gives $f(q_1,q_2^{-1})=\chi^4(h)$. Similarly $f(q_1^{-1},q_2)=\chi^5(h)$ for all $q_2^2<q_1$ with $q_1\neq q_2^3$.

It remains to discuss the cases $q_1=q_2$ and $q_1=q_2^3$. Let us briefly outline the working involved. Consider the $q_1=q_2$ case (the $q_1=q_2^3$ case is similar). 
 The integrands of $I_1,\ldots,I_6$ are still nonsingular on $\TT$, and the contours in the integrals $I_3,I_4$ and $I_6$ can all be shifted to $\TT$ without encountering any poles. This leaves $I_1,I_2$ and $I_5$ to consider. Writing $f(t)=g(t)/d(t)$ the integrands of $I_1$, $I_2$ and $I_5$ are (respectively)
\begin{align*}
\frac{qs^2(1-s^2)g(q^{\frac{3}{2}}s,q^{-1})}{(1-q)(1-q^{-\frac{1}{2}}s^{-1})^2(1+q^{-\frac{1}{2}}s^{-1})(1-q^{-\frac{5}{2}}s^{-1})(1-q^{-\frac{1}{2}}s)^2(1+q^{-\frac{1}{2}}s)(1-q^{-\frac{5}{2}}s)}\\
\frac{s^4(1-s^2)g(q^{-1},q^{\frac{1}{2}}s)}{(1-q)(1-q^{-\frac{3}{2}}s^{-3})(1-q^{-1}s^{-2})(1-q^{-\frac{3}{2}}s^{-1})(1-q^{-\frac{3}{2}}s^3)(1-q^{-1}s^2)(1-q^{-\frac{3}{2}}s)}\\
\frac{q^{-\frac{1}{2}}s(1-s^2)g(q^{-\frac{3}{2}}s^{-1},q^{\frac{1}{2}}s)}{(q-1)(1-q^{-\frac{1}{2}}s^{-1})^2(1+q^{-\frac{1}{2}}s^{-1})(1-q^{-\frac{5}{2}}s^{-1})(1-q^{-\frac{1}{2}}s)^2(1+q^{-\frac{1}{2}}s)(1-q^{-\frac{5}{2}}s)}.
\end{align*}
The relevant poles are at $s=q^{-\frac{1}{2}}$ (a double pole for $I_1$, $I_2$, and $I_5$), $s=q^{-\frac{1}{2}}$ (a single pole for $I_1$, $I_2$, and $I_5$), $s=\omega^{\pm1}q^{-\frac{1}{2}}$ (single poles for $I_2$ only), and $s=q^{-\frac{5}{2}}$ (a single pole for $I_1$ only). The residue contributions from $s=-q^{-\frac{1}{2}}$ make up the $\chi_-^7(h)$ term, the contributions from $s=\omega^{\pm1}q^{-\frac{1}{2}}$ give the $\chi^8(h)$ term, and the contribution from $s=q^{-\frac{5}{2}}$ gives the $\chi^3(h)$ term. All that remains is to analyse the contribution from the double poles of each integral at $s=q^{-\frac{1}{2}}$.

We claim that the combined residue contribution from the point $s=q^{-\frac{1}{2}}$ is
\begin{align}\label{eq:entangled}
R_1+R_2+R_5=\frac{q(q-1)^3}{6(q+1)^2(q^3-1)}\left(\chi_+^7(h)+2\chi^4(h)\right).
\end{align}
We do not have a conceptual proof of this fact, but it can be obtained by direct calculation as follows. As in Remark~\ref{rem:explicit} the functions $g(t)=g_t(h)$ can be explicitly computed (since one only needs to know the values $g_t(T_w)$ for $w\in W_0$, as $g_t(T_wx^{\lambda})=t^{\lambda}g_t(T_w)$). Then the residue contributions can be explicitly calculated (making $12$ separate calculations, one for each $h=T_wx^{\lambda}$ with $w\in W_0$). On the other hand, using the explicit matrices (Example~2 of Section~\ref{sect:examples}) for the representations $\pi_+^7$ and $\pi^4$ one can compute the expression $\chi_+^7(T_wx^{\lambda})+2\chi^4(T_wx^{\lambda})$ and compare. This completes the proof. 
\end{proof}

\subsection{The $\tilde{BC}_2(q_0,q_1,q_2)$ algebras with $L=Q$}

The root system is $R=\pm\{\alpha_1,\alpha_2,\alpha_1+\alpha_2,\alpha_1+2\alpha_2,2\alpha_2,2(\alpha_1+\alpha_2)\}$, giving dual root system $R^{\vee}=\pm\{\alpha_1^{\vee},\alpha_2^{\vee},\alpha_1^{\vee}+\alpha_2^{\vee},2\alpha_1^{\vee}+\alpha_2^{\vee},\alpha_2^{\vee}/2,\alpha_1^{\vee}+\alpha_2^{\vee}/2\}$. The affine Hecke algebra has generators $T_1,T_2,x_1=x^{\alpha_1^{\vee}}$ and $x_2=x^{\alpha_2^{\vee}/2}$ with relations
\begin{align*}
T_1^2&=1+(q_1^{\frac{1}{2}}-q_1^{-\frac{1}{2}})T_1& T_1x_1&=x_1^{-1}T_1+(q_1^{\frac{1}{2}}-q_1^{-\frac{1}{2}})(1+x_1)\\
T_2^2&=1+(q_2^{\frac{1}{2}}-q_2^{-\frac{1}{2}})T_2 & T_2x_2&=x_2^{-1}T_2+(q_2^{\frac{1}{2}}-q_2^{-\frac{1}{2}})x_2+(q_0^{\frac{1}{2}}-q_0^{-\frac{1}{2}})\\
T_1T_2T_1T_2&=T_2T_1T_2T_1 & T_2x_1&=x_1x_2^2T_2^{-1}-(q_0^{\frac{1}{2}}-q_0^{-\frac{1}{2}})x_1x_2\\
 x_1x_2&=x_2x_1 &
T_1x_2&=x_1x_2T_1^{-1}.
\end{align*}

Let $\pi_t=\mathrm{Ind}_{\CC[Q]}^{\scH}(\CC v_t)$ be the principal series representation of $\scH$ with central character $t=(t_1,t_2)\in(\CC^{\times})^2$, where $\CC v_t$ is the $1$-dimensional representation of $\CC[Q]$ with $x_1\cdot v_t=t_1v_t$ and $x_2\cdot v_t=t_2v_t$.

Let $\scH_1$ be the subalgebra generated by $T_1,x_1,x_2$ and let $\scH_2$ be the subalgebra generated by $T_2,x_1,x_2$. Let $s\in\CC^{\times}$, and let $\pi_s^{1}=\mathrm{Ind}_{\scH_1}^{\scH}(\CC u_s^1)$ and $\pi_s^j=\mathrm{Ind}_{\scH_2}^{\scH}(\CC u_s^j)$ ($j=2,3,4$) be the $4$-dimensional representations induced from the $1$-dimensional representation $\CC u_s^1$ of $\scH_1$ and the $1$-dimensional representations $\CC u_s^j$ ($j=2,3,4$) of $\scH_2$
given by
\begin{align*}
T_1\cdot u_s^1&=-q_1^{-\frac{1}{2}}u_s^1 & x_1\cdot u_s^1&=q_1^{-1}u_s^1 & x_2\cdot u_s^1&=q_1^{\frac{1}{2}}su_s^1\\
T_2\cdot u_s^2&=-q_2^{-\frac{1}{2}}u_s^2 & x_1\cdot u_s^2&=q_0^{\frac{1}{2}}q_2^{\frac{1}{2}}su_s^2 & x_2\cdot u_s^2&=q_0^{-\frac{1}{2}}q_2^{-\frac{1}{2}}u_s^2\\
T_2\cdot u_s^3&=-q_2^{-\frac{1}{2}}u_s^3 & x_1\cdot u_s^3&=q_0^{-\frac{1}{2}}q_2^{\frac{1}{2}}su_s^3 & x_2\cdot u_s^3&=-q_0^{\frac{1}{2}}q_2^{-\frac{1}{2}}u_s^3\\
T_2\cdot u_s^4&=q_2^{\frac{1}{2}}u_s^4 & x_1\cdot u_s^4&=q_0^{\frac{1}{2}}q_2^{-\frac{1}{2}}su_s^4 & x_2\cdot u_s^4&=-q_0^{-\frac{1}{2}}q_2^{\frac{1}{2}}u_s^4.
\end{align*}

Let $\pi^j$ $(j=5,\ldots,11)$ be the $1$-dimensional representations of $\scH$ with
\begin{align*}
\pi^5&=(-q_1^{-\frac{1}{2}},-q_2^{-\frac{1}{2}},q_1^{-1},q_0^{-\frac{1}{2}}q_2^{-\frac{1}{2}})&
\pi^6&=(-q_1^{-\frac{1}{2}},-q_2^{-\frac{1}{2}},q_1^{-1},-q_0^{\frac{1}{2}}q_2^{-\frac{1}{2}})\\
\pi^7&=(q_1^{\frac{1}{2}},-q_2^{-\frac{1}{2}},q_1,q_0^{-\frac{1}{2}}q_2^{-\frac{1}{2}})&\pi^8&=(q_1^{\frac{1}{2}},-q_2^{-\frac{1}{2}},q_1,-q_0^{\frac{1}{2}}q_2^{-\frac{1}{2}})\\
\pi^9&=(-q_1^{-\frac{1}{2}},q_2^{\frac{1}{2}},q_1^{-1},q_0^{\frac{1}{2}}q_2^{\frac{1}{2}})&
\pi^{10}&=(-q_1^{-\frac{1}{2}},q_2^{\frac{1}{2}},q_1^{-1},-q_0^{-\frac{1}{2}}q_2^{\frac{1}{2}})\\
\pi^{11}&=(q_1^{\frac{1}{2}},q_2^{\frac{1}{2}},q_1,-q_0^{-\frac{1}{2}}q_2^{\frac{1}{2}}),
\end{align*}
where in each case we list the quadruples $(\pi^j(T_1),\pi^j(T_2),\pi^j(x_1),\pi^j(x_2))$.

Let $\pi^{12}=M_J(s)$, $\pi^{13}=M_J(t)$, and $\pi^{14}=M_J(u)$ be the $3$-dimensional representations with
\begin{align*}
(s^{\alpha_1^{\vee}},s^{\alpha_2^{\vee}/2})=(q_1^{-1},q_0^{\frac{1}{2}}q_2^{\frac{1}{2}}),\quad (t^{\alpha_1^{\vee}},t^{\alpha_2^{\vee}/2})=(q_1^{-1},-q_0^{-\frac{1}{2}}q_2^{\frac{1}{2}}),\quad (u^{\alpha_1^{\vee}},u^{\alpha_2^{\vee}/2})=(q_1^{-1},-q_0^{\frac{1}{2}}q_2^{-\frac{1}{2}})
\end{align*}
and $J=\{\alpha_2\}$. We assume that $q_1\neq q_0q_2$ and $q_1^2\neq q_0q_2$ for $\pi^{12}$, that $q_1\neq q_0^{-1}q_2$ and $q_1^2\neq q_0^{-1}q_2$ for $\pi^{13}$, and that $q_1\neq q_0q_2^{-1}$ and $q_1^2\neq q_0q_2^{-1}$ for $\pi^{14}$, so that
$N(s)=N(t)=N(u)=\{\alpha_1,\alpha_2\}$ and $D(s)=D(t)=D(u)=\emptyset$,
and hence $F_J(s)=F_J(t)=F_J(u)=\{s_2,s_1s_2,s_2s_1s_2\}$.

Finally, let $\pi^{15}=M_J(t)$ and $\pi^{16}=M_J(u)$ be the $2$-dimensional representations with
\begin{align*}
(t^{\alpha_1^{\vee}},t^{\alpha_2^{\vee}/2})&=(-q_0,q_0^{-\frac{1}{2}}q_2^{-\frac{1}{2}})&(u^{\alpha_1^{\vee}},u^{\alpha_2^{\vee}/2})&=(-q_2,q_0^{-\frac{1}{2}}q_2^{-\frac{1}{2}})&J&=\emptyset.
\end{align*}

\begin{thm} For all $h\in\scH$ we have
\begin{align*}
\Tr(h)&=\frac{1}{8q_1^2q_2^2}\iint_{\TT^2}\frac{\chi_t(h)}{|c(t)|^2}\,dt+C_6\int_{\TT}\frac{\chi_s^1(h)}{|c_1(s)|^2}\,ds+C_7\int_{\TT}\frac{\chi_s^2(h)}{|c_2(s)|^2}\,ds+C_8\int_{\TT}\frac{\chi(h)}{|c_3(s)|^2}\,ds\\
&\quad+C_1\chi^5(h)+|C_2|\times\begin{cases}
\pi^7(h)&\textrm{if $q_1<q_0^{\frac{1}{2}}q_2^{\frac{1}{2}}$}\\
\pi^{12}(h)&\textrm{if $q_0^{\frac{1}{2}}q_2^{\frac{1}{2}}<q_1<q_0q_2$}\\
\pi^9(h)&\textrm{if $q_0q_2<q_1$}
\end{cases}\,\,+\,\,\begin{cases}X_1&\textrm{if $q_0<q_2$}\\
X_2&\textrm{if $q_2<q_0$}
\end{cases}
\end{align*}
where $\chi(h)=\chi_s^3(h)$ if $q_0<q_2$ and $\chi(h)=\chi_s^4(h)$ if $q_2<q_0$, and where
\begin{align*}
X_1&=|C_3|\chi^{15}(h)+|C_4|\chi^6(h)+|C_5|\times\begin{cases}
\chi^8(h)&\textrm{if $q_1<q_0^{-\frac{1}{2}}q_2^{\frac{1}{2}}$}\\
\chi^{13}(h)&\textrm{if $q_0^{-\frac{1}{2}}q_2^{\frac{1}{2}}<q_1<q_0^{-1}q_2$}\\
\pi^{10}(h)&\textrm{if $q_0^{-1}q_2<q_1$}
\end{cases}\\
X_2&=|C_3|\chi^{16}(h)+|C_5|\pi^{10}(h)+|C_4|\times\begin{cases}
\pi^{11}(h)&\textrm{if $q_1<q_0^{\frac{1}{2}}q_2^{-\frac{1}{2}}$}\\
\pi^{14}(h)&\textrm{if $q_0^{\frac{1}{2}}q_2^{-\frac{1}{2}}<q_1<q_0q_2^{-1}$}\\
\pi^6(h)&\textrm{if $q_0q_2^{-1}<q_1$},
\end{cases}
\end{align*}
with $c(t),c_1(s),c_2(s),c_3(s),C_1,\ldots,C_8$ as in Appendix~\ref{app:BC2Q}. 
\end{thm}

\begin{proof}
The series $G_t(h)$ converges for $|t_1|<q_1^{-1}$ and $|t_2|<q_0^{-\frac{1}{2}}q_2^{-\frac{1}{2}}$, and so writing $f(t)=f_t(h)$ we have
$$
\Tr(h)=\frac{1}{q_1^2q_2^2}\int_{q_1^{-1}a\TT}\int_{q_0^{-\frac{1}{2}}q_2^{-\frac{1}{2}}b\TT}\frac{f(t)}{c(t)c(t^{-1})}\,dt_2dt_1
$$
whenever $0<a,b<1$. We choose $a$ and $b$ both very close to $0$, and choose $0<c<1$ very close to~$1$. 
The $t_2$-poles of the integrand between the contour $q_0^{-\frac{1}{2}}q_2^{-\frac{1}{2}}b\TT$ and $c\TT$ are at $t_2=q_0^{-\frac{1}{2}}q_2^{-\frac{1}{2}}$, at $t_2=-q_0^{\frac{1}{2}}q_2^{-\frac{1}{2}}$ (if $q_0<q_2$) and at $t_2=-q_0^{-\frac{1}{2}}q_2^{\frac{1}{2}}$ (if $q_2<q_1$). Thus we can shift the $t_2$-contour to $c\TT$ at the cost of residue contributions from the above points. 
Now interchange the order of integration in the double integral. The $t_1$-poles of the integrand between $q_1^{-1}a\TT$ and $\TT$ are at $t_1=q_1^{-1}$, $t_1=q_1^{-1}t_2^{-2}$, $t_1=q_0^{-\frac{1}{2}}q_2^{-\frac{1}{2}}t_2^{-1}$, $t_1=-q_0^{\frac{1}{2}}q_2^{-\frac{1}{2}}t_2^{-1}$ (if $q_0<q_2$) and $t_1=-q_0^{-\frac{1}{2}}q_2^{\frac{1}{2}}t_2^{-1}$ (if $q_2<q_0$). Computing the associated residues gives
$$
\Tr(h)=\frac{1}{q_1^2q_2^2}\int_{\TT^2}\frac{f(t)}{|c(t)|^2}\,dt+I_1+I_2+I_3+I_4+\begin{cases}I_5+I_6&\textrm{if $q_0<q_2$}\\
I_5'+I_6'&\textrm{if $q_2<q_0$},
\end{cases}
$$
where
\begin{align*}
I_1&=\frac{q_0q_2-1}{q_1^2q_2(q_0+1)(q_2+1)}\int_{q_0^{-\frac{1}{2}}q_1^{-1}q_2^{-\frac{1}{2}}a\TT}\frac{f(q_0^{\frac{1}{2}}q_2^{\frac{1}{2}}s,q_0^{-\frac{1}{2}}q_2^{-\frac{1}{2}})}{c_2(s)c_2(s^{-1})}\,ds&&\textrm{where $s=q_0^{-\frac{1}{2}}q_2^{-\frac{1}{2}}t_1$}\\
I_2&=\frac{q_1-1}{q_1q_2^2(q_1+1)}\int_{q_1^{-\frac{1}{2}}c\TT}\frac{f(q_1^{-1},q_1^{\frac{1}{2}}s)}{c_1(s)c_1(s^{-1})}\,ds&&\textrm{where $s=q_1^{-\frac{1}{2}}t_2$}\\
I_3&=\frac{q_1-1}{q_1q_2^2(q_1+1)}\int_{q_1^{\frac{1}{2}}c\TT}\frac{f(s^{-2},q_1^{-\frac{1}{2}}s)}{c_1(s)c_1(s^{-1})}\,ds&&\textrm{where $s=q_1^{\frac{1}{2}}t_2$}\\
I_4&=\frac{q_0q_2-1}{q_1^2q_2(q_0+1)(q_2+1)}\int_{c\TT}\frac{f(q_0^{-\frac{1}{2}}q_2^{-\frac{1}{2}}s^{-1},s)}{c_2(s)c_2(s^{-1})}\,ds&&\textrm{where $s=t_2$}\\
I_5&=\frac{q_2-q_0}{q_1^2q_2(q_0+1)(q_2+1)}\int_{q_0^{\frac{1}{2}}q_1^{-1}q_2^{-\frac{1}{2}}a\TT}\frac{f(q_0^{-\frac{1}{2}}q_2^{\frac{1}{2}}s,-q_0^{\frac{1}{2}}q_2^{-\frac{1}{2}})}{c_3(s)c_3(s^{-1})}\,ds&&\textrm{where $s=q_0^{\frac{1}{2}}q_2^{-\frac{1}{2}}t_1$}\\
I_5'&=\frac{q_0-q_2}{q_1^2q_2(q_0+1)(q_2+1)}\int_{q_0^{-\frac{1}{2}}q_1^{-1}q_2^{\frac{1}{2}}a\TT}\frac{f(q_0^{\frac{1}{2}}q_2^{-\frac{1}{2}}s,-q_0^{-\frac{1}{2}}q_2^{\frac{1}{2}})}{c_3(s)c_3(s^{-1})}\,ds&&\textrm{where $s=q_0^{-\frac{1}{2}}q_2^{\frac{1}{2}}t_1$}\\
I_6&=\frac{q_2-q_0}{q_1^2q_2(q_0+1)(q_2+1)}\int_{c\TT}\frac{f(q_0^{\frac{1}{2}}q_2^{-\frac{1}{2}}s^{-1},-s)}{c_3(s)c_3(s^{-1})}\,ds&&\textrm{where $s=-t_2$}\\
I_6'&=\frac{q_0-q_2}{q_1^2q_2(q_0+1)(q_2+1)}\int_{c\TT}\frac{f(q_0^{-\frac{1}{2}}q_2^{\frac{1}{2}}s^{-1},-s)}{c_3(s)c_3(s^{-1})}\,ds&&\textrm{where $s=-t_2$}.
\end{align*}
Now shift the contours in all integrals $I_j$, $I_j'$ to $\TT$. We omit the details of this long calculation.
\end{proof}

\begin{appendix}

\newpage

\section{Constants and $c$-functions}

We write $\sigma_1(x)=1+x$ and $\sigma_2(x)=1+x+x^2$.

\subsection{$\tilde{C}_2(q_1,q_2)$ algebras with $L=Q$}\label{app:C2Q}

$$
c(t)=\frac{(1-q_1^{-1}t_1^{-1})(1-q_1^{-1}t_1^{-1}t_2^{-2})(1-q_2^{-1}t_2^{-1})(1-q_2^{-1}t_1^{-1}t_2^{-1})}{(1-t_1^{-1})(1-t_1^{-1}t_2^{-2})(1-t_2^{-1})(1-t_1^{-1}t_2^{-1})}
$$
$$
c_1(s)=\frac{(1+q_1^{-\frac{1}{2}}s^{-1})(1-q_1^{-\frac{1}{2}}q_2^{-1}s^{-1})(1-q_1^{\frac{1}{2}}q_2^{-1}s^{-1})}{(1-s^{-2})(1-q_1^{\frac{1}{2}}s^{-1})} \quad
c_2(s)=\frac{(1-q_1^{-1}q_2^{-1}s^{-1})(1-q_1^{-1}q_2s^{-1})}{(1-s^{-1})(1-q_2s^{-1})}
$$

\begin{align*}
A&=\frac{(q_1q_2-1)(q_1q_2^2-1)}{\sigma_1(q_1)\sigma_1(q_2)^2\sigma_1(q_1q_2)}&B&=\frac{2q_2(q_1-1)^2}{\sigma_1(q_2)^2\sigma_1(q_1q_2^{-1})\sigma_1(q_1q_2)}&
C&=\frac{(q_1q_2^{-1}-1)(1-q_1q_2^{-2})}{\sigma_1(q_1)\sigma_1(q_2^{-1})^2\sigma_1(q_1q_2^{-1})}.
\end{align*}

\subsection{$\tilde{C}_2(q_1,q_2)$ algebras with $L=P$}\label{app:C2P}

$$
c(t)=\frac{(1-q_1^{-1}t_1^{-2}t_2^2)(1-q_1^{-1}t_2^{-2})(1-q_2^{-1}t_1t_2^{-2})(1-q_2^{-1}t_1^{-1})}{(1-t_1^{-2}t_2^2)(1-t_2^{-2})(1-t_1t_2^{-2})(1-t_1^{-1})}
$$
$$
c_1(s)=\frac{(1+q_1^{-\frac{1}{2}}s^{-1})(1-q_1^{-\frac{1}{2}}q_2^{-1}s^{-1})(1-q_1^{\frac{1}{2}}q_2^{-1}s^{-1})}{(1-s^{-2})(1-q_1^{\frac{1}{2}}s^{-1})} \quad
c_2(s)=\frac{(1-q_1^{-1}q_2^{-1}s^{-2})(1-q_1^{-1}q_2s^{-2})}{(1-s^{-2})(1-q_2s^{-2})}
$$

\subsection{$\tilde{G}_2(q_1,q_2)$ algebras with $L=Q$}\label{app:G2Q}

\begin{align*}
c(t)&=\frac{(1-q_1^{-1}t_1^{-1})(1-q_1^{-1}t_1^{-2}t_2^{-3})(1-q_1^{-1}t_1^{-1}t_2^{-3})(1-q_2^{-1}t_2^{-1})(1-q_2^{-1}t_1^{-1}t_2^{-2})(1-q_2^{-1}t_1^{-1}t_2^{-1})}{(1-t_1^{-1})(1-t_1^{-2}t_2^{-3})(1-t_1^{-1}t_2^{-3})(1-t_2^{-1})(1-t_1^{-1}t_2^{-2})(1-t_1^{-1}t_2^{-1})}\\
c_1(s)&=\frac{(1-q_1^{-\frac{1}{2}}\omega s^{-1})(1-q_1^{-\frac{1}{2}}\omega^{-1}s^{-1})(1-q_2^{-1}s^{-2})(1-q_1^{-\frac{1}{2}}q_2^{-1}s^{-1})(1-q_1^{\frac{1}{2}}q_2^{-1}s^{-1})}{(1-s^{-2})(1-q_1^{-\frac{1}{2}}s^{-1})(1-q_1^{\frac{1}{2}}s^{-3})}\\
c_2(s)&=\frac{(1-q_1^{-1}s^{-2})(1-q_1^{-1}q_2^{-\frac{3}{2}}s^{-1})(1-q_1^{-1}q_2^{\frac{3}{2}}s^{-1})}{(1-s^{-2})(1-q_2^{\frac{3}{2}}s^{-1})(1-q_2^{\frac{1}{2}}s^{-1})}
\end{align*}
\begin{align*}
A&=\frac{(q_1q_2^2-1)(q_1^2q_2^3-1)}{\sigma_1(q_1)\sigma_1(q_2)\sigma_2(q_2)\sigma_2(q_1q_2)}&
B_{\pm}&=\frac{q_1(q_1-1)(q_2-1)}{2\sigma_1(q_1)\sigma_1(q_2)\sigma_2(\pm\sqrt{q_1/q_2})\sigma_2(\pm\sqrt{q_1q_2})}\\
C&=\frac{q_2(q_1-1)(q_1^3-1)}{\sigma_2(q_2)\sigma_2(q_1q_2^{-1})\sigma_2(q_1q_2)}&
D&=\frac{(1-q_1q_2^{-2})(q_1^2q_2^{-3}-1)}{\sigma_1(q_1)\sigma_1(q_2^{-1})\sigma_2(q_2^{-1})\sigma_2(q_1q_2^{-1})}.
\end{align*}

\subsection{$\tilde{BC}_2(q_0,q_1,q_2)$ algebras with $L=Q$}\label{app:BC2Q}

\begin{align*}
c(t)&=\frac{(1-q_1^{-1}t_1^{-1})(1-q_1^{-1}t_1^{-1}t_2^{-2})(1-a^{-1}t_1^{-1}t_2^{-1})(1+b^{-1}t_1^{-1}t_2^{-1})(1-a^{-1}t_2^{-1})(1+b^{-1}t_2^{-1})}{(1-t_1^{-1})(1-t_1^{-1}t_2^{-2})(1-t_1^{-2}t_2^{-2})(1-t_2^{-2})}\\
c_1(s)&=\frac{(1-q_0^{-\frac{1}{2}}q_1^{-\frac{1}{2}}q_2^{-\frac{1}{2}}s^{-1})(1+q_0^{\frac{1}{2}}q_1^{-\frac{1}{2}}q_2^{-\frac{1}{2}}s^{-1})(1-q_0^{-\frac{1}{2}}q_1^{\frac{1}{2}}q_2^{-\frac{1}{2}}s^{-1})(1+q_0^{\frac{1}{2}}q_1^{\frac{1}{2}}q_2^{-\frac{1}{2}}s^{-1})}{(1-s^{-2})(1-q_1s^{-2})}\\
c_2(s)&=\frac{(1+q_0^{\frac{1}{2}}q_2^{-\frac{1}{2}}s^{-1})(1-q_0^{-\frac{1}{2}}q_1^{-1}q_2^{-\frac{1}{2}}s^{-1})(1-q_0^{\frac{1}{2}}q_1^{-1}q_2^{\frac{1}{2}}s^{-1})}{(1-s^{-2})(1-q_0^{\frac{1}{2}}q_2^{\frac{1}{2}}s^{-1})}\\
c_3(s)&=\frac{(1+q_0^{-\frac{1}{2}}q_2^{-\frac{1}{2}}s^{-1})(1-q_0^{\frac{1}{2}}q_1^{-1}q_2^{-\frac{1}{2}}s^{-1})(1-q_0^{-\frac{1}{2}}q_1^{-1}q_2^{\frac{1}{2}}s^{-1})}{(1-s^{-2})(1-q_0^{-\frac{1}{2}}q_2^{\frac{1}{2}}s^{-1})}
\end{align*}
where $a=q_0^{\frac{1}{2}}q_2^{\frac{1}{2}}$ and $b=q_0^{-\frac{1}{2}}q_2^{\frac{1}{2}}$. 
\begin{align*}
C_1&=\frac{(q_0q_1q_2-1)(q_0q_1^2q_2-1)}{\sigma_1(q_0)\sigma_1(q_1)\sigma_1(q_2)\sigma_1(q_0q_1)\sigma_1(q_1q_2)}&C_3&=\frac{(q_2-q_0)(q_0q_2-1)}{\sigma_1(q_0q_1^{-1})\sigma_1(q_1^{-1}q_2)\sigma_1(q_0q_1)\sigma_1(q_1q_2)}
\end{align*}
and $C_2=-C_1(q_0^{-1},q_1,q_2^{-1})$, $C_4=C_1(q_0^{-1},q_1,q_2)$ and $C_5=-C_1(q_0,q_1,q_2^{-1})$. Finally,
\begin{align*}
C_6&=\frac{q_1-1}{2q_1q_2^2(q_1+1)}&C_7&=\frac{q_0q_2-1}{2q_1^2q_2(q_0+1)(q_2+1)}&C_8&=\frac{|q_2-q_0|}{2q_1^2q_2(q_0+1)(q_2+1)}.
\end{align*}

\end{appendix}

\end{document}